\documentclass{amsart}
\usepackage{graphicx} 
\usepackage{amsthm}
\usepackage{amsmath}
\usepackage{amssymb}
\usepackage{amsfonts}
\usepackage{amscd}
\usepackage{comment}
\usepackage{tikz}
\usepackage{bm}

\newcommand{\Kahler}{K\"{a}hler}

\newtheorem{thm}{Theorem}
\newtheorem{prop}{Proposition}
\newtheorem{lem}{Lemma}
\newtheorem*{thm*}{Theorem}

\theoremstyle{remark}

\newtheorem{dfn}{\textbf{Definition}}
\newtheorem{eg}{Example}

\newtheorem{Conj*}{\textbf{Conjecture}}

\everymath{\displaystyle}

\title[on asymptotic behavior of the second chern forms]%
{ON ASYMPTOTIC BEHAVIOR OF THE SECOND CHERN FORMS ON DEGENERATING K\"{A}HLER-EINSTEIN SURFACES}
\author{Itsuki Tazoe} 
\address{Department of Mathematics, Kyoto University, Kyoto 606-8285. JAPAN}
\email{tazoe.itsuki.52n@st.kyoto-u.ac.jp}

\date{}
\begin{document}

\maketitle

\begin{abstract}
    We study an asymptotic behavior of the second Chern forms of canonical metrics on a degenerating family of Kähler surfaces with the central fibre having ADE-singularities. We investigate a function on the unit disc defined by fiber integrals of the forms with a smooth test function on the family. We show a lower bound of the Hölder exponent of the function at the origin. Our main results consists of two cases: one is a bound of Hölder exponent along a line for cscK-metrics, using Biquard-Rollin's a priori estimates for cscK-metrics, and the other is a bound of Hölder exponent at the origin for Ricci-flat metrics.
\end{abstract}

\section{Introduction}

In \cite{King}, J.R.King studied behavior of flat proper family of analytic spaces as currents:
\begin{thm}[Theorem 3.3.2. in \cite{King}]
    Let $f:X \to Y$ be a proper flat map of complex analytic spaces with irreducible fibers of dimension $d$.
    Then for a continuous $(d,d)$ form $\varphi$ on $X$, the function $\Phi: Y \to \mathbb{C}$ defined by 
    \begin{equation}
        \Phi(y) := \int_{f^{-1}(y)} \varphi
    \end{equation}
    is a continuous function on $Y$.
\end{thm}
In \cite{Bar}, D. Barlet developed the behavior of the fibre integral for a one parameter degenerating family of manifolds with assumption on the smoothness of the integrant:
\begin{thm}[Th\'eor\`eme 1 in \cite{Bar}]
    Let $\mathcal{X}$ be an irreducible and reduced analytic space of dimension $n+1$. Let $f : \mathcal{X} \to \Delta$ be a holomorphic surjection onto the unit disc $\Delta \subset \mathbb{C}$ and let $X_t := f^{-1}(t) \subset \mathcal{X}$ be the fibre on $t \in \Delta$. Then for every compact subset $K \subset \mathcal{X}$, there are rational numbers $r_1, \ldots , r_k \in [0,2) \cap \mathbb{Q}$ such that for every smooth $(n,n)$-form $\phi$ on $\mathcal{X}$ supported on $K$, its fibre integration
    \begin{equation*}
        F_\phi(t) = \int_{X_t} \phi 
    \end{equation*}
    admits an asymptotic expansion 
    \begin{equation*}
        F_\phi(t) \sim \sum_{\substack{ r= r_1,\ldots, r_k \\
        j=1,\ldots,n \\
        (m,n) \in \mathbb{N}^2}} T^{r, j}_{m,n} (\phi) t^m \overline{t}^n |t|^r (\log(|t|))^j
    \end{equation*}
    as $t \to 0$, where $T^{r,j}_{m,n}$ are $(1,1)$-currents on $\mathcal{X}$.
\end{thm}
 As a special case of the above theorem, we have asymptotic expansion of forms coming from a characteristic form of a metric on $\mathcal{X}$.

 In the present paper, we study asymptotic behavior of fiber integrals \textit{without assumptions on regularity} of the integrant a priori but taking the integrant \textit{canonically}.
 Explicitly, we study asymptotic behavior of characteristic forms of a \textit{ canonical metric} on the relative tangent bundle $T \mathcal{X}/\Delta$ which is not necessarily coming from a smooth metric on $T\mathcal{X}$ and not necessarily smooth around the origin $0 \in \Delta$ a priori.
 
Studying \textit{canonical metrics} on complex manifolds is one of the central problems in complex geometry.
 For Riemann surfaces, the most fundamental situation, we have a unique (up to constant multiplication) metric on every compact Riemann surface with constant Gauss curvature.
 For higher dimensional case, the existence theorems by Aubin and Yau \cite{Au}\cite{Yau}, Yau\cite{Yau}, and Chen-Donaldson-Sun\cite{CDS} guarantee existence of K\"ahler-Einstein metrics, i.e. K\"ahler metrics with constant Ricci curvatures, of general type, Calabi-Yau and (K-stable) Fano manifolds respectively.

 It is natural to investigate degeneration of K\"ahler-Einstein metrics in an algebro-geometric setting. 
 A pioneering work by R. Kobayashi and A. Todorov \cite{KT} investigated degeneration of polarized K3 surfaces, an important class of 2-dimensional Calabi-Yau manifold. 
 In \cite{KT}, it was proved that if polarized K3 surfaces \textit{converges} to a polarized K3 surface possibly singular in an \textit{algebro-geometric sense}, then they also converges to the same K3 surface in the \textit{Gromov-Hausdorff sense} (see \cite{KT} Theorem 8 for the precise statement). 
 Rather Riemannian Geometric context, Anderson \cite{An}, Nakajima \cite{Na}, Bando-Kasue-Nakajima\cite{BKN} etc. studied limits of 4-dimensional Einstein manifolds (see section \ref{sec:Non-coll.} for more details). 
 In \cite{DS1}\cite{DS2}, it was shown that if a convergent sequence of polarized K\"ahler manifolds with uniformly bounded volumes, diameters and Ricci curvatures satisfies a \textit{non-collapsing condition} then its limit is a normal projective variety at worst klt singularities.

 Two different type of \textit{canonical metrics} are treated:
 One is constant scaler curvature K\"ahler metrics (cscK metrics for short) and the other is Ricci-flat metrics.
 The main results of the present paper are the following. 
\begin{thm*}[Theorem \ref{thn:main} in section \ref{sec:Main}]
Let $(\mathcal{X} \to \Delta, \{\Omega_t\}_{t \in \Delta})$ be a smoothing family of compact surface $(X_0, \omega_0)$ with ADE-singularities (see section \ref{sec:BR} for its precise meaning), where $\Delta \subset \mathbb{C}$ is the unit disc, and $f$ be a smooth function on $\mathcal{X}$ with its support containing only $x_0\in X_0 \cong 0 \in \mathbb{C}^2/\Gamma$. Assume that the family admits a simultaneous resolution $\widetilde{\mathcal{X}} \to \Delta_d$ after taking a base change by the ramified covering $\Delta_d \to \Delta$ of degree $d$ and it is non-degenerate at $x_0$ along a real ray $\Delta\cap (0,1)$ (see definition \ref{dfn:non-deg}). Furthermore, we assume that $X_t$ admits a cscK metric $g_t$ with its K\"ahler form $\omega_t$ in $\Omega_t$. Consider a function 
\begin{equation*}
F(t) = \int_{X_t} f_t c_2(g_t) 
\end{equation*}
on $\Delta \cap (0,1)$, where $g_t$ is the cscK metric on $X_t$ with its \Kahler\ class $\Omega_t$, $c_2(g_t)$ is the second Chern form of $g_t$ on $TX_t$ and $f_t = f |_{X_t}$. Then $F$ extends to a H\"{o}lder continuous function on $\Delta\cap [0,1)$ with H\"older exponent  at least $\frac{1}{d}$.
\end{thm*}
The above theorem is proven by using an a priori estimate of cscK metrics by Biquard-Rollin \cite{BR}.
A family of K\"ahler metrics which approximates a family of cscK metrics with prescribed K\"ahler classes is constructed and the difference of the approximating metrics and the cscK metrics are estimated by the a priori estimate of cscK equation in \cite{BR}.
Using the results, the behavior of $c_2(g_t)$ can be investigated by using the approximating metrics and the a priori estimate.

In the above theorem, we only know asymptotic behavior of $c_2(g_t)$ \textit{along the real ray}.
To investigate asymptotic behavior \textit{on the hole disc}, we further assume that the family is K\"ahler-Einstein.
For a K\"ahler--Einstein case, we consider a family of polerized K3 surfaces $(\mathcal{X}, \mathcal{L}) \to \Delta$.
For each fiber $(X_t, L_t)$, let $\hat{g}_t$ be the Ricci flat K\"ahler metric with its K\"ahler form $\hat{\omega}_t \in c_1(L_t)$. 
To study behavior of the second chern forms $c_2(\hat{g}_t)$, an \textit{almost Ricci flat} K\"ahler metrics $\tilde{g}_t$ on $X_t$ which approximates $\hat{g}_t$ well are constructed (see section \ref{sec:K3} for details). 
And then, the metric $\tilde{g}_t$ approximates $\hat{g}_t$ not only its metric structure, but also its Ricci curvature. 
Then the main theorem in section \ref{sec:K3} is the following. 
\begin{thm*}[Theorem \ref{thm:main2} in section \ref{sec:K3}]
    Let $(\mathcal{X}, \mathcal{L}) \to \Delta$ be a flat proper family of polarized K3 surfaces over the unit disc $\Delta \subset \mathbb{C}$ such that the fibre $(X_t,L_t)$ on $t \in \Delta$ is smooth for $t \neq 0$ and singular for $t =0$ with ADE singularities. Assume that the family $\mathcal{X} \to \Delta$ admits the minimal simultaneous resolution after taking the base change $\Delta_d \to \Delta$ and non-degenerate at a singularity $x_0 \in X_0 \cong \mathbb{C}^2/\Gamma$. Then for any smooth function $f$ on $\mathcal{X}$ with its support around the singularity $x_0 \in X_0 \cong 0 \in \mathbb{C}^2/\Gamma$, the function
    \begin{equation*}
        F(t) = \int_{X_t} f_t c_2(\hat{g}_t)
    \end{equation*}
    on $\Delta\backslash \{0\}$ extends to a H\"older continuous function on $\Delta$ with H\"older exponent at least $\frac{1}{d}$.
\end{thm*} 

In section \ref{sec:pre}, we recall some preliminary facts, the theory of Kronheimer's ALE gravitational instanton, non-collapsing limits of Einstein 4-manifolds and Biquard-Rollin's a priori estimate for constant scaler curvature equation for a family of surfaces. Section 3 is the proof of the main theorem for the cscK metric case. 
The main theorem for Ricci flat K\"ahler metrics on a family of polarized K3 surfaces is proven in section \ref{sec:K3}.

\textbf{Acknowledgements}. The author is grateful to Professor K.-I. Yoshikawa and Professor Y. Odaka, his supervisors, for their help and encouragement which lead him to complete this master thesis. The author is also grateful to E.Inoue and S.Katayama for suggestions on some technical problems.

\section{Preliminaries}\label{sec:pre}

\subsection{ALE Gravitational Instanton}\label{sec:ALE}

\textit{ALE Gravitational Instantons} are families of 4-dimensional complete hyperk\"ahler manifolds that play an important role in the study of Gromov-Hausdorff limits of \Kahler-Einstein surfaces. They appear as \textit{bubbling limits} of non-collapsing Gromov-Hausdorff limits of \Kahler-Einstein surfaces (for details, see section \ref{sec:Non-coll.}).

\begin{dfn}\label{def:ALE}
    Let $(X,g,I,J,K)$ be a 4-dimensional complete hyperk\"ahler manifold. $(X,g,I,J,K)$ is an \textit{ALE gravitational instanton} (of order 4) if there is a compact subset $K \subset X$ and a diffeomorphism 
    \begin{equation}
        \phi : (\mathbb{R}^4 \backslash B(0,R))/ \Gamma \to X\backslash K
    \end{equation}
    for a suitable $R>0$ and a finite subgroup $\Gamma \subset \mathrm{SL}(2;\mathbb{C})$ such that 
        \begin{equation}
        \begin{aligned}
            \|\nabla^k (\phi^*g - g^{Euc}) \|_{g^\mathrm{Euc}} &= O(r^{-4-k}),\\
            \| \nabla^k( \phi^* \omega_I - \omega_I^{\mathrm{Euc}} ) \|_{g^\mathrm{Euc}} & = O(r^{-4-k}),\\
             \| \nabla^k( \phi^* \omega_J - \omega_J^{\mathrm{Euc}} ) \|_{g^\mathrm{Euc}} & = O(r^{-4-k}),\\
              \| \nabla^k( \phi^* \omega_K - \omega_K^{\mathrm{Euc}} ) \|_{g^\mathrm{Euc}} & = O(r^{-4-k})
        \end{aligned}    
        \end{equation}
    as $r \to \infty$ for any $k \in \mathbb{N}$, where $r$ is the Euclidean distance from the origin. 
\end{dfn}
Note that $\mathbb{C}^2$ admits a natural hyperk\"ahler structure via an isomorphism
\begin{equation}
    \begin{aligned}
        \mathbb{H} & \to \mathbb{C}^2 \\
        a + bi + cj + dk & \mapsto (a+ bi, c + di ).
    \end{aligned}
\end{equation}
 To state the classification theorem of ALE gravitational instantons by Kronheimer \cite{Kr1} \cite{Kr2}, we review some properties of ADE singularities. 
 Let $\Gamma \subset \textrm{SL}(2;\mathbb{C})$ be a finite subgroup and let $Y$ denote the underlying differentiable manifold of the minimal resolution $\widetilde{\mathbb{C}^2/\Gamma}$ of the singularity $0 \in \mathbb{C}^2/\Gamma$ with the exceptional divisor $E = \bigcup_{i=1}^n E_i$. 
 Then, it is well-known that the homology classes of the irreducible components $\{[E_i]\}_{i=1}^n$ forms a basis of $H_2(Y, \mathbb{Z})$ as a $\mathbb{Z}$-module. 
 Furthermore, by the Poincar\'e-Lefschetz duality, we may define an intersection pairing on $H_2(Y, \mathbb{Z})$ which induces an isomorphism 
 \begin{equation}
     H_2 (Y,\mathbb{Z}) \cong H^2(Y,\mathbb{Z}).
 \end{equation}
 The lattice $H_2(Y,\mathbb{Z})$ is isometric to the lattice defined by a Dynkin diagram of one of the ADE-type. It is also important that  $H_2(Y, \mathbb{Z})$ is isometric to the root lattice $\mathfrak{h}_\mathbb{Z}$ of the simple complex Lie algebra $\mathfrak{g}$ which corresponds to the same Dynkin diagram by a linear map sending $\{[E_i]\}$ to simple roots (see \cite{Slo1} for more details). 
 Now, we can state the classification theorem that was proved in \cite{Kr1}\cite{Kr2}:

    \begin{thm}[Theorem 1.1, Theorem 1.2, Theorem 1.3 in \cite{Kr1}\cite{Kr2}]
        (i) Let $\bm{\kappa}=(\kappa_1,\kappa_2 , \kappa_3)$ be a triple of elements in $H^2(Y, \mathbb{R})$ that satisfies the following property:
            \begin{equation}\label{cond:generality}
                \text{for any}\ e \in H_2(Y, \mathbb{Z})\ \text{with}\ (e,e) = -2,\ \kappa_i(e) \neq 0\  
                \text{for at least one}\  i=1,2,3
            \end{equation}
        Then there exists a quadruple $(g_{\bm{\kappa}},I_{\bm{\kappa}},J_{\bm{\kappa}},K_{\bm{\kappa}})$ such that $Y_{\bm{\kappa}} :=(Y,g_{\bm{\kappa}},I_{\bm{\kappa}},J_{\bm{\kappa}},K_{\bm{\kappa}})$ is an ALE gravitational instanton and 
        \begin{equation}
        ([\omega_{I_{\bm{\kappa}}}], [\omega_{J_{\bm{\kappa}}}], [\omega_{K_{\bm{\kappa}}}]) = \bm{\kappa},    
        \end{equation}
        where $\omega_{I_{\bm{\kappa}}}, \omega_{J_{\bm{\kappa}}}, \omega_{K_{\bm{\kappa}}}$ are the \Kahler\ forms of the three complex structures.

        (ii) Any ALE gravitational instanton $Y'$ is diffeomorphic to $\widetilde{\mathbb{C}^2/\Gamma}$ for some $\Gamma \subset \mathrm{SL}(2; \mathbb{C})$ and the triple of the \Kahler\ classes of $Y'$ satisfies the property (\ref{cond:generality}) in (i). 
         
        (iii) Let $(Y_1,g_1,I_1,J_1,K_1), (Y_2,g_2,I_2,J_2,K_2)$ be ALE gravitational instantons. If there is a diffeomorphism $Y_1 \to Y_2$ which respects the triples of the \Kahler\ classes, then $(Y_1, g_1,I_1,J_1,K_1)$ and $(Y_2, g_2, I_2,J_2,K_2)$ are isomorphic as hyperk\"ahler manifolds, i.e. there exists a diffeomorphism $f:Y_1 \to Y_2$ such that 
        \begin{equation}
            f^*g_2 = g_1, f^*\omega_{I_2}= \omega_{I_1}, f^* \omega_{J_2} = \omega_{J_1} , f^*\omega_{K_2} = \omega_{K_1}.
        \end{equation}
    \end{thm}
To state some basic properties of ALE gravitational instantons, following \cite{Kr1}, we recall an outline of the construction of ALE gravitational instantons on $Y$.
For a finite subgroup $\Gamma \subset \mathrm{SL}(2; \mathbb{C})$, define $M$ as the $\Gamma$ invariant points of the tensor product $\mathbb{C}^2 \otimes_\mathbb{C} \mathrm{End} (\mathbb{C}[\Gamma])$: 
\begin{equation*}
  M := \{ v \in \mathbb{C}^2 \otimes_\mathbb{C} \mathrm{End}(\mathbb{C}[\Gamma]) \mid g\cdot v =v\} \subset \mathbb{C}^2 \otimes_\mathbb{C} \mathrm{End}(\mathbb{C}[\Gamma]),
\end{equation*}
where $\mathbb{C}^2$ is viewed as the natural $\Gamma$ module (recall that $\Gamma$ is a subgroup of $\mathrm{SL}(2;\mathbb{C})$) and $\mathbb{C}[\Gamma]$ is the group ring of $\Gamma$. 
Then one can equip $M$ a hyperk\"ahler structure and a Lie group action $F \curvearrowright M$ which preserves the hyperk\"ahler structure, where $F$ is a suitable sub Lie group of $\mathrm{U}(\mathbb{C}[\Gamma])$ (see \cite{Kr1} for the definitions of $F$ and the action $F \curvearrowright M$). 
Also, one can construct a \textit{hyperk\"ahler moment map} 
\begin{equation}\label{eq:mu}
\mu: M \to \mathfrak{h}_{\mathbb{R}} \otimes \mathbb{R}^3
\end{equation}
(see \cite{Kr1} section 2 for their concrete construction), where $\mathfrak{h}_{\mathbb{R}} = \mathfrak{h}_\mathbb{Z} \otimes \mathbb{R}$ is the real part of a Cartan subalgebra of the corresponding ADE type of $\Gamma$. In \cite{Kr1}, it was shown that the quotient $Y_\zeta := \mu^{-1}(\zeta)/ F$ is non singular if $\zeta \in (\mathfrak{h}_\mathbb{R} \otimes \mathbb{R}^3)^o$, where
\begin{equation}\label{eq:pd}
    (\mathfrak{h}_{\mathbb{R}} \otimes \mathbb{R}^3)^o = (\mathfrak{h}_{\mathbb{R}} \otimes \mathbb{R}^3) \backslash \bigcup_{\theta\ \text{is a root}} H_\theta \otimes \mathbb{R}^3.
\end{equation}
Here, $H_\theta := \theta^{\perp} \subset \mathfrak{h}_\mathbb{R}$ denotes the hyperplane cut out by $\theta$ and $\theta \in \mathfrak{h}_{\mathbb{Z}}$ is a root, by definition, if $(\theta, \theta) =-2$.
Hence, as the hyperk\"ahler structure on $M$ is preserved by the group action $F \curvearrowright M$, the quotient $Y_\zeta = \mu^{-1}(\zeta)/F$ becomes a hyperk\"ahler manifold.  
Let $(Y_\zeta, g_\zeta, I_\zeta, J_\zeta, K_\zeta)$ denote the hyperk\"ahler structure on $Y_\zeta$. 
In the following, we regard $Y_\zeta$ as a complex surface by the complex structure $I_\zeta$ and identify $\mathfrak{h}_\mathbb{R} \otimes \mathbb{R}^3 
 \cong (\mathfrak{h}_\mathbb{R})^{\oplus 3}$ as $ \mathfrak{h}_\mathbb{R} \oplus \mathfrak{h}_\mathbb{C}$ by 
\begin{equation}
\begin{aligned}
    (\mathfrak{h}_\mathbb{R})^{\oplus 3} &\to \mathfrak{h}_\mathbb{R} \oplus \mathfrak{h}_\mathbb{C} \\
    (\zeta_1, \zeta_2, \zeta_3) &\mapsto (\zeta_1, \zeta_2 + \sqrt{-1} \zeta_3) = (\zeta_r, \zeta_c).
\end{aligned}
\end{equation}
     In \cite{Kr1}, it was also proved that the quotient $Y_{(0,\zeta_c)}$ becomes a (possibly singular) affine complex variety for any $\zeta_c \in \mathfrak{h}_\mathbb{C}$ and that, in particular, $Y_{(0,0)}$ is biholomorphic to $\mathbb{C}^2/\Gamma$.
     Furthermore, if $\zeta_r \in \mathfrak{h}_\mathbb{R}$ is not perpendicular to any root in $\mathfrak{h}_\mathbb{Z}$, there exists a holomorphic map
 \begin{equation}\label{eq:res.ofALE}
    \lambda_\zeta : Y_{(\zeta_r, \zeta_c)} \to Y_{(0,\zeta_c)}
 \end{equation}
 that gives the minimal resolution of $Y_{(0,\zeta_c)}$. For a given non-singular $Y_\zeta$, by a suitable choice of a coordinate of $\mathbb{R}^3$, we may assume that $\zeta_1$ is not perpendicular to any root. Then by using (\ref{eq:res.ofALE}) repeatedly, we have the following:
 \begin{equation}\label{eq:res.of.ALE2}
     Y_\zeta \xrightarrow{\lambda_\zeta} Y_{(\zeta_1, \zeta_2, 0)} \xrightarrow{\lambda_{(\zeta_1, \zeta_2,0)}} Y_{(\zeta_1, 0, 0)} \xrightarrow{\lambda_{(\zeta_1, 0, 0)}} Y_{(0,0,0)} \cong \mathbb{C}^2/ \Gamma.
 \end{equation}
 Since $Y_{(\zeta_1, 0,0)}$ is naturally diffeomorphic to $Y$, we see that any non-singular $Y_\zeta$ is naturally diffeomorphic to $Y$. The diffeomorphism $Y_\zeta \cong Y$ induced by (\ref{eq:res.of.ALE2}) depends on the choice of orthonormal basis of $\mathbb{R}^3$, however, we may choose the diffeomorphisms in a canonical way in the following sense:
 \begin{thm}[\cite{Kr1} Section 3 and Section 4, namely from Corollary 3.12.]\label{thm:identification}
    There exist identifications 
    \begin{equation}\label{eq:identify}
        \begin{aligned}
            H_2 (Y, \mathbb{Z}) &\cong \mathfrak{h}_\mathbb{Z},\\
            \Lambda_\zeta : Y &\to Y_\zeta,
        \end{aligned}
    \end{equation}
     such that for any $\zeta \in (\mathfrak{h}_\mathbb{R} \otimes \mathbb{R}^3)^o$, $\Lambda_\zeta^* ([\omega_i]) = \zeta_i \in \mathfrak{h}_\mathbb{R}$, where $(\omega_1, \omega_2, \omega_3) = (\omega_I, \omega_J, \omega_K)$ are the \Kahler\ forms of $Y_\zeta$.
     Furthermore, $\{(\Lambda_\zeta^*g_\zeta, \Lambda_\zeta^* I_\zeta, \Lambda_\zeta^*J_\zeta, \Lambda_\zeta^*K_\zeta)\}_{\zeta \in (\mathfrak{h} \otimes \mathbb{R}^3)^o}$ is a family of hyperk\"ahler structures on $Y$ depending on $\zeta$ smoothly (this is the hyperk\"ahler structure stated in Theorem 1).
 \end{thm}
 Now, we are ready to state properties of $Y_\zeta$ which we will use in the paper:
\begin{itemize}
    \item (\cite{Kr1} Proposition 3.14. and its proof) any non-singular $Y_\zeta$ is an ALE gravitational instanton. Namely, by the composition
     \begin{equation}\label{eq:rho}
        \rho_\zeta : Y_\zeta \xrightarrow{\Lambda_\zeta^{-1}} \widetilde{\mathbb{C}^2/\Gamma} 
        \xrightarrow{\text{blow down}} \mathbb{C}^2/\Gamma
     \end{equation}
    we have a diffeomorphism 
    \begin{equation}\label{eq:unif.ofALE}
       \varphi_\zeta: (\mathbb{C}^2 \backslash \overline{B(0;R)})/\Gamma \to Y_\zeta \backslash K
    \end{equation}
    for any $R >0$ and a suitable compact subset $K \subset Y_\zeta$ ($B(0,R)$ is the open ball centered at the origin with diameter $R$). Under this diffeomorphism, we have the expansion of the hyperk\"ahler metric $g_\zeta$ and its \Kahler\ form $\omega_{I_\zeta}$ on $\mathbb{C}^2 \backslash \overline{B(0,R)}$:
    \begin{align}
        \varphi_\zeta^* g_\zeta = g^\mathrm{Euc} + \sum_{i = 2}^{k-1} g_\zeta^{(i)} r^{-2i} + O(r^{2k})\label{eq:asy.metic},\\
        \varphi_\zeta^* \omega_{I_\zeta} = \omega^{\mathrm{Euc}} + \sum_{i=2}^{k-1} \omega_{I_\zeta}^{(i)} r^{-2i} + O(r^{2k})\label{eq:asy.form},
    \end{align}
    where $g_\zeta^{(i)}, \omega_{I_\zeta}^{(i)}$ are homogeneous polynomials of degree $i$ in $\zeta$ and $g^{\mathrm{Euc}}, \omega^{\mathrm{Euc}}$ are the standard \Kahler\ metric and its \Kahler\ form on $\mathbb{C}^2$ (where $r(x) := d(0,x)$ is the radius function on $\mathbb{C}^2$).

    \item (\cite{Kr1} section 3 and 4)
    \begin{equation}\label{eq:family}
    \mathcal{Y} = \bigcup_{\zeta \in \mathfrak{h}_\mathbb{C}} Y_{(0,\zeta)}
    \end{equation}
    admits a natural complex structure such that the natural projection $\mathcal{Y} \to \mathfrak{h}_\mathbb{C}$ is holomorphic. This family $\mathcal{Y} \to \mathfrak{h}_\mathbb{C}$ coincides with the base change of the versal deformation of the singularity $0 \in \mathbb{C}^2/\Gamma$ by the Weyl covering $\mathfrak{h}_\mathbb{C} \to \mathfrak{h}_\mathbb{C}/W$, where $W$ is the Weyl group of $\mathfrak{h}$. Furthermore, for any $\zeta_r \in \mathfrak{h}_\mathbb{R}$ that is not perpendicular to any roots of $\mathfrak{h}$ and for any family $\mathcal{Y}_{(0,\zeta)} \to \Delta$ over the unit disc defined by a holomorphic map $\zeta : \Delta \to \mathfrak{h}_\mathbb{C}$, the family
    \begin{equation}\label{eq:sim.resol.byALE}
     \mathcal{Y}_{(\zeta_r, \zeta)} = \bigcup Y_{(\zeta_r, \zeta(t))} \to \Delta
    \end{equation}
     (it also admits a natural complex structure) gives the minimal simultaneous minimal resolution of $\mathcal{Y}_{(0,\zeta)}$ (it is induced by applying (\ref{eq:res.ofALE}) fiberwise).
    
    \item (\cite{Kr1} Proof of Corollary 3.12) The natural $\mathbb{R}_{>0}$-action on $M$ 
    \begin{equation*}
       \begin{aligned}
             M &\to M\\
        v &\mapsto \alpha v
       \end{aligned}
    \end{equation*}
    induces a morphism
    \begin{equation}\label{eq:dil}
        H_{(\alpha, \zeta)} : Y_\zeta \to Y_{\alpha^2 \zeta}
    \end{equation}
     which induces an isomorphism 
     \begin{equation}
     Y_\zeta \cong \frac{1}{\alpha^2} Y_{\alpha^2\zeta} = \left(Y, \frac{1}{\alpha^2} g_{\alpha^2 \zeta}, I_{\alpha^2 \zeta}, J_{\alpha^2 \zeta}, K_{\alpha^2 \zeta}  \right)
     \end{equation}
     as hyperk\"ahler manifolds.
    Furthermore, an open embedding $(\mathbb{C}^2\backslash \{0\})/\Gamma \to Y_\zeta$ induced by (\ref{eq:rho}) is $\mathbb{R}_{>0}$ equivariant. Namely, for any $y \in Y_\zeta$ such that $r(y) >0$, 
    \begin{equation}\label{eq:equivariance}
    \rho_{\alpha^2\zeta} \circ H_{(\alpha, \zeta)}(y) = \alpha \rho_\zeta(y).
    \end{equation}
    In particular, if $f = O(r^k)$ on $Y$ as $r \to \infty$, $H_{(\alpha, \zeta)}^*(f) = O(\alpha^k r^k)$ as $r \to \infty$.

    \item (\cite{Kr1} section 4) The family $\mathcal{Y} \to \mathfrak{h}_\mathbb{C}$ (see (\ref{eq:family})) admits an embedding 
    \begin{equation}\label{eq:embedding}
    \iota : \mathcal{Y} \to \mathbb{C}^3 \times \mathfrak{h}_\mathbb{C}    
    \end{equation}
    such that the projection $\mathcal{Y} \to \mathfrak{h}_\mathbb{C}$ is induced by the natural projection $\mathbb{C}^3 \times \mathfrak{h}_\mathbb{C} \to \mathfrak{h}_\mathbb{C}$. Furthermore, this is $\mathbb{C}^*$ equivariant, i.e. there is a weighted $\mathbb{C}^*$-action on $\mathbb{C}^3$ with weight $k = (k_1,k_2,k_3)$ and a $\mathbb{C}^*$-action $H_{(\alpha, \zeta)} : Y_\zeta \to Y_{\alpha^2\zeta} $ extending the $\mathbb{R}_{>0}$-action above such that $\iota \circ H_{(\alpha,\zeta)}(y) = \alpha\cdot \iota(y)$ (the weight of the $\mathbb{C}^*$-action on $\mathfrak{h}_\mathbb{C}$ is $2$) for any $y \in Y_\zeta \subset \mathcal{Y}$.
\end{itemize} 
 In the rest of the paper, we fix a differentiable manifold $Y$ which underlies $\widetilde{\mathbb{C}^2/\Gamma}$ and identify $Y_\zeta$ as $(Y, \Lambda_\zeta^*g_\zeta, \Lambda_\zeta^* I_\zeta, \Lambda_\zeta^*J_\zeta, \Lambda_\zeta^*K_\zeta)$ (under this identification, we simply write $Y_\zeta = (Y, g_\zeta, I_\zeta, J_\zeta, K_\zeta)$). In particular, for a family $\mathcal{Y}_{(0,\zeta_c)} \to \Delta$, the simultaneous resolution $\mathcal{Y}_\zeta \to \mathcal{Y}_{(0,\zeta_c)}$ given by (\ref{eq:sim.resol.byALE}) yields a natural trivialization 
 \begin{equation}\label{eq:triv.byALE}
     Y \times \Delta \cong \mathcal{Y}_\zeta
 \end{equation}
 as families of differentiable manifolds induced by the fiberwise identification (\ref{eq:identify}) (c.f. Theorem \ref{thm:identification}). Also, we regard $H_{(\alpha, \zeta)}$ as a $\mathbb{R}_{>0}$ action on $Y$ parameterized by $\zeta \in (\mathfrak{h} \otimes \mathbb{R}^3)^o$. We fix these identifications (namely, identifications from (\ref{eq:identify}) to (\ref{eq:triv.byALE})) in the rest of the paper. 
 
The next proposition plays an important role for the proof of the main result. We also use $r$ for the pull-back on $Y$ of the radius function $ r(p) = d(0,p) $ of $\mathbb{C}^2$.
\begin{lem}\label{lem:cpx.str}
    Regard $Y_\zeta$ as a complex manifold by the complex structure $I_\zeta$. Then for any $\zeta$ varying in a compact subset in $(\mathfrak{h}_\mathbb{R} \otimes \mathbb{R}^3)^o$, the following hold:
    \begin{itemize}
        \item
        \begin{equation}\label{est:cpx.str.}
        I_\zeta = I_0 + \sum_{i=2}^{k-1} I_\zeta^{(i)} r^{-2i} + O(r^{-2k})
        \end{equation}
        as $r \to \infty$, where $I^{(i)}_\zeta$ is a matrix valued homogeneous polynomial in $\zeta$ of degree $i$ under the trivialization of the tangent bundle by the standard coordinate $(z_0,w_0)$ on $\mathbb{C}^2 \backslash \{(0,0)\}$,
        \item there is a system of frames $\{e_\zeta^1, e_\zeta^2\}$ of $T^{(1,0)}Y_\zeta $ on $\{y \in Y \mid r \geq 1\}$ satisfying the following estimates:
        \begin{align}
    e_\zeta^1 = \frac{\partial}{\partial z_0}+ O(r^{-4}) &,\ 
    e_\zeta^2 =\frac{\partial}{\partial w_0} + O(r^{-4}) , \label{est:frames1} \\ 
    (e_\zeta^1)^* = dz_0 + O(r^{-4}) &,\ 
    (e_\zeta^2)^* = dw_0 + O(r^{-4}),\label{est:frames2}
\end{align}
where $(e_\zeta^i)^*$ is the dual frame of $e^i_\zeta$.
    \end{itemize}
\end{lem}

\begin{proof}
    Let 
    \begin{equation*}
        I_\zeta \mid_{r=1} = I_0 + \sum_\nu I_{\nu} \zeta^\nu  
    \end{equation*}
    be the expansion on the unit sphere in $\mathbb{C}^2$ with respect to $\zeta$ ($\nu$ is a multi-index). Then we have $I_\zeta = H_{r^{-1}}^* (I_0 + \sum_\nu I_\nu \zeta^{\nu})= I_0 + I_\zeta^{(2)} r^{-2} + O(r^{-4})$, where $I_\zeta^{(2)} = \sum_{|\nu| = 1} I_\nu \zeta^{\nu}$. By the estimates (\ref{eq:asy.metic}) and (\ref{eq:asy.form}), we have uniform estimates 
    \begin{align*}
        g_\zeta = g^\mathrm{Euc} + O(r^{-4}), \\
        \omega_{I_\zeta} = \omega^{\mathrm{Euc}} + O(r^{-4})
    \end{align*}
    on a compact subset of $(\mathfrak{h} \otimes \mathbb{R}^3)^o$.
    As the complex structure $I_\zeta$ is characterized by a equality
    \begin{equation*}
        g_\zeta (I_\zeta u,v) = \omega_{I_\zeta} (u,v) 
    \end{equation*}
    for any tangent vectors $u,v$, $I_\zeta$ satisfies the estimate 
    \begin{equation*}
        I_\zeta = I_0 + \sum_{i=2}^{k-1} I_\zeta^{(i)} r^{-2i} + O(r^{-2k})
    \end{equation*}
    as desired. For the second claim, take $e^i_\zeta$ to be 
    \begin{align*}
        e^1_\zeta = \frac{1}{2} \left( \frac{\partial}{\partial z_0} - \sqrt{-1} I_\zeta \frac{\partial}{\partial z_0}\right),\ 
        e^2_\zeta = \frac{1}{2} \left( \frac{\partial}{\partial w_0} - \sqrt{-1} I_\zeta \frac{\partial}{\partial w_0}\right).
    \end{align*}
    Then, $\{e_\zeta^1, e_\zeta^2\}$ is a frame of $T^{(1,0)}Y_\zeta$ on the domain and satisfies the estimates (\ref{est:frames1}) and (\ref{est:frames2}) by its construction and the estimate (\ref{est:cpx.str.}) of the complex structure $I_\zeta$.
\end{proof}
\subsection{Non-collapsing limits of K\"{a}hler-Einstein surfaces}\label{sec:Non-coll.}

 Let $\{(X_i,g_i)\}_{i=1}^\infty$ be a sequence of $n$-dimensional Riemannian manifolds. 
 When we consider its limit, it is necessary to assume some \textit{boundedness conditions} for the sequence.
 If the diameters and the Ricci curvatures of $\{(X_i, g_i)\}_{i=1}^\infty$ are uniformly bounded and there exists a constant $C>0$ such that $\mathrm{Vol}_{g_i} (B(p,r)) > C r^n$ for any $p \in X_i, r > 0$, we call the sequence \textit{non-collapsing sequence} and its limit (if it exists) \textit{non-collapsing limit}.
  A non-collapsing limit of polarized K\"{a}hler-Einstein manifolds is a normal algebraic variety which has at worst klt singularities \cite{DS1}\cite{DS2}. 
  In dimension two, by Anderson\cite{An}, Bando\cite{Ba}, Bando-Kasue-Nakajima\cite{BKN}, and Nakajima\cite{Na}, a non-collapsing limit of compact hyperk\"ahler surfaces has at worst quotient singularities, and its bubbling limits at singularities are ALE of order $4$:

\begin{dfn}
    Let $\{(X_i,g_i)\}_{i=1}^\infty$ be a sequence of Riemannian manifolds that converges to a metric space $(X_\infty, d_\infty)$ in the Gromov-Hausdorff sense. 
    Then its \textit{bubbling limit at $x_\infty \in X_\infty $} is a limit of a sequence $\{(X_i, r_i g_i, x_i)\}$ in the pointed Gromov-Hausdorff sense (take a suitable subsequence if necessary) where $\{r_i\}$ is a divergent sequence of positive real numbers and $\{x_i\}$ is a sequence of points such that $x_i \in X_i$ and $x_i \to x_\infty$.
\end{dfn}

\begin{thm}[\cite{An}, \cite{Na} for (i), \cite{BKN} for (ii) and \cite{Ba} for (iii)]\label{thm:non-coll}
   (i) Let $\{(X_i, g_i)\}$ be a sequence of real 4-dimensional compact Einstein manifolds such that diameters  and $L^2$-energies $\int_{X_i} \|\mathrm{Rm}_{g_i}\|^2_{g_i} \mathrm{Vol}_{g_i}$ are uniformly bounded from above and volumes are uniformly bounded from below. 
   Then the sequence  $\{(X_i, g_i)\}$ contains a convergent subsequence that converges to an Einstein orbifold $(X_\infty, g_\infty)$ in the following sense:\\
For any compact subset $K\subset X_\infty^{reg}$, $k>0$ and $\varepsilon>0$, there exists $N \in \mathbb{N}$ such that for any $i >N$, we have open immersions $\chi_i : U_i \to X_i$, where $U_i$ are open neighborhoods of $K$ in $X_{\infty}^{reg}$ such that $\|g_\infty - \chi_i^* g_i\|_{C^k(K)} < \varepsilon$ .
   
   (ii)Let $\{(X_i, g_i)\}$ be a sequence of 4-dimensional Einstein manifolds with uniformly bounded diameters, $\int_{X_i} \|\mathrm{Rm}_{g_i}\|^2 \mathrm{Vol}_{g_i}$ and volumes, which converges to an Einstein orbifold $(X_\infty,g_\infty)$. 
   Then the set of isometric classes of bubbling limits at a singularity $x \in X_\infty$ is a finite set. Furthermore, if the sequence $\{(X_i,g_i)\}$ consists of \Kahler\ (hyperk\"ahler) manifolds, then $(X_\infty, g_\infty)$ is a \Kahler\ (resp. hyperk\"ahler) orbifold.

   (iii) Under the same assumption as in (ii), one has the following formula:
   \begin{equation}
       \lim_{i \to \infty} \|\mathrm{Rm}_{g_i} \|^2 \mathrm{Vol}_{g_i} \geq \|\mathrm{Rm}_{g_\infty}\|^2\mathrm{Vol}_{g_\infty}  +\sum_{x \in X_\infty^{sing}} h_x \delta_x
    \end{equation}
    
    as measures, where 
    \begin{equation}
        \begin{aligned}
            h_x &= \sum_{(M,g): \text{bubbling at}\ x} e_{\mathrm{orb}} (M), \\
            e_{\mathrm{orb}} (M) &= \frac{1}{8 \pi^2}\int_{M} \|\mathrm{Rm}_g \|^2 \mathrm{Vol}_g.
        \end{aligned}
    \end{equation}
    and $\delta_x$ is the Dirac measure at $x$. 
    The precise meaning of the above limit is the following:\\
    For any sequence $\{f_i\}$ of smooth functions $f_i \in C^\infty(X_i)$ that converges to a smooth function $f_\infty$ on $X_\infty$ (the convergence is defined by the completely same manner as in (i)),
    \begin{equation}
         \lim_{i \to \infty} \int_{X_i} f_i \|\mathrm{Rm}_{g_i}\|^2 \mathrm{Vol}_{g_i} \geq \int_{X_\infty} f_\infty\|\mathrm{Rm}_{g_\infty}\|^2 \mathrm{Vol}_{g_\infty} + \sum f_\infty(x) h_x.
    \end{equation}
   
\end{thm}

\subsection{Biquard-Rollin's a priori estimate}\label{sec:BR}
    In this section, we recall Biquard-Rollin's a priori estimate \cite{BR} for constant scaler curvature \Kahler\ metrics (or, cscK-metrics for short) on a degenerating family of surfaces.
    
   Set-ups: Let $(X_0,g_0)$ be a pair consisting of a compact complex surface $X_0$ with ADE singularities and a cscK orbifold metric $g_0$ on $X_0$. Let $\omega_0$ denote the \Kahler\ form of $g_0$. 
   In order to state the a priori estimate, we need to define the notion of a \textit{non-degenerate smoothing} of $(X_0, g_0)$ (this non-degeneracy condition is a terminology used only \cite{BR}, not generally used). 
   Let $\mathcal{X} \to \Delta$ be a flat deformation of $X_0$ over the unit disc of the complex plane. 
   Assume that the fibers $X_t$ are smooth for $t \in \Delta^* = \Delta\backslash \{0\}$. 
   Take a family of \Kahler\ classes $\{\kappa_t\}_{t \in \Delta}$ such that:
   \begin{itemize}
       \item $\kappa_t$ is a \Kahler\ class on $X_t$,
       \item On a minimal simultaneous resolution $\widetilde{\mathcal{X}} \to \Delta_d$ with a suitable base change 
       \begin{equation}
       \begin{aligned}
           \Delta_d &\to \Delta \\
           t &\mapsto t^d,
       \end{aligned}
       \end{equation}
       the pull-back of $\{\kappa_t\}$ on $\widetilde{\mathcal{X}}$ defines a smooth curve in $H^2(X, \mathbb{R})$ converging to the pull-back of $[\omega_0] =: \kappa_0$, where $X$ is the differentiable manifold underlying the fibers of $\widetilde{\mathcal{X}} \to \Delta_d$. Here, $\Delta_d$ is also the unit disc in $\mathbb{C}$ and we use the subscript `` $d$ " to distinguish the original family and its base change.
   \end{itemize}
   For each singularity $(x \in X_0) \cong (0\in \mathbb{C}^2/\Gamma)$, by restricting the given family to a suitable neighbourhood $\mathcal{U}$ of $x \in \mathcal{X}$, we have a flat deformation of the ADE singularity $x$.
   Hence we have a holomorphic map $\phi : \Delta \to \mathfrak{h}_\mathbb{C}/ W$, where $\mathfrak{h}_\mathbb{C}$ is the complexification of the Cartan sub-algebra and $W$ is the Weyl group corresponding to $\Gamma$. Let $\pi: \mathcal{U}_d \to \Delta_d$ be the family obtained from $\mathcal{U} \to \Delta$ by the base-change $\Delta_d \to \Delta$, the ramified covering of degree $d$ ramified at the origin. Again, the family $\mathcal{U}_d \to \Delta_d$ is induced by a holomorphic map $\psi : \Delta_d \to \mathfrak{h}_\mathbb{C}/W$. Then, by the theory of deformations of ADE singularities, we have a holomorphic map
   \begin{equation}\label{eq:period}
       \zeta_c : \Delta_d \to \mathfrak{h}_\mathbb{C}
   \end{equation}
   such that the following diagram is commutative:
\begin{center}
$
\begin{CD}
 \Delta_d @>\zeta_c >> \mathfrak{h}_\mathbb{C}\\
@VVV @VVV\\
     \Delta  @>\phi >> \mathfrak{h}_\mathbb{C}/W
\end{CD}
$.
\end{center}
Here, the morphism $\Delta_d \to \Delta$ is the ramified covering from the unit disc to itself of degree d and the composition $\Delta_d \to \mathfrak{h}_\mathbb{C} \to \mathfrak{h}_\mathbb{C}/W$ is equal to $\psi$. The family $\mathcal{U}_d \to \Delta_d$ admits the minimal simultaneous resolution $\tilde{\pi} : \widetilde{\mathcal{U}} \to \Delta_d$. We may define a smooth map $\zeta_r: \Delta_d \to \mathfrak{h}_\mathbb{R}$ by $\zeta_r(t) = \kappa_t\mid_{\widetilde{U}_t}$ under a suitable identification $H^2(\widetilde{U}_t, \mathbb{R}) \cong \mathfrak{h}_\mathbb{R}$, where $\widetilde{U}_t$ is the fiber of $\widetilde{\mathcal{U}}$ on $t \in \Delta_d$
(Note that $R^2\tilde{\pi}_*\mathbb{R}$ is trivial on $\Delta_d$ and is identified with the constant sheaf with stalk $\mathfrak{h}_\mathbb{R}$. See also Theorem \ref{thm:identification}).

\begin{dfn}\label{dfn:non-deg}
    In the same situation as above, the family $(\mathcal{X} \to \Delta, \{\kappa_t\})$ is \textit{non-degenerate at x} along a real ray $\Delta_d \cap \mathbb{R}_{\geq 0}$ if $\zeta = \zeta_r \oplus \zeta_c : \Delta_d \to \mathfrak{h}_\mathbb{R} \oplus \mathfrak{h}_\mathbb{C}$ satisfies the following condition:
    as $\zeta(0)=0$, we can express
    \begin{equation}\label{eq:order of zeta}
        \zeta(t) = t^p \dot{\zeta} + O(t^{p+1})
    \end{equation}
    with $\dot{\zeta} \neq 0$. Then $\dot{\zeta} \in (\mathfrak{h}_\mathbb{R} \otimes \mathbb{R}^3)^o$ (cf. (\ref{eq:pd})) and $p \leq d$.
\end{dfn}

\begin{eg}
For a given family $(\mathcal{X} \to \Delta, \{\kappa_t\})$, assume that the central fiber $X_0$ has an $A_n$ type singularity $x \in X$. Then its flat deformation given by a suitable open neighborhood $\mathcal{U} \to \Delta$ is isomorphic to an open subset of
\begin{equation*}
    \{ xy + z^{n+1} + a_2(t) z^{n-1} + \cdots + a_{n+1} (t) =0\} \subset \mathbb{C}^3 _{(x,y,z)} \times \Delta
\end{equation*}
where $a_i(t)$ are analytic in $t$. Furthermore, $\mathcal{U}_d \to \Delta_d$ is isomorphic to an open subset of
\begin{equation*}
    \{ xy + \prod_{i=1}^{n+1} (z-h_i(t)) =0 \} \subset \mathbb{C}^3_{(x,y,z)} \times \Delta_d
\end{equation*}
where $h_i$ are analytic in $t$. For the family $\mathcal{U}_d \to \Delta_d$, we have
\begin{equation*}
    \zeta_c(t) = (h_1(t),\ldots,h_{n+1}(t)) \in \mathfrak{h}_\mathbb{C}.
\end{equation*}
For simplicity, take a family of \Kahler\ classes such that $\zeta_r \in \mathfrak{h}_\mathbb{R}$ is identically zero  (for example, take $(\mathcal{X} \to \Delta, \{\kappa_t\})$ consisting of a family of K3 surfaces and the first Chern classes of an ample line bundle $\mathcal{L}$ on $\mathcal{X}$ and restrict the family to a suitable neighborhood $\mathcal{U}$ of a singularity $x\in X_0$). Then, the family is non-degenerate if and only if $\zeta_c(t) = t^p(\dot{h}_1(t), \ldots, \dot{h}_{n+1}(t))$ with $p \leq d$ and $(\dot{h}_1(0), \ldots, \dot{h}_{n+1}(0)) \notin \bigcup_{i < j} H_{e_i, - e_j}$, where $H_{e_i - e_j}$ is the hyperplane orthogonal to $e_i -e_j = (0,\ldots, \overset{i}{1}, \ldots, \overset{j}{-1}, \ldots, 0)$.
\end{eg}

Let $(\mathcal{X} \to \Delta, \{\kappa_t\})$ be a non-degenerate family. We construct a family of  \textit{almost cscK-metrics} as follows:\\
For simplicity, we assume that $X_0$ has only one singularity $(x \in X_0) \cong (0\in \mathbb{C}^2/\Gamma )$. Take a diffeomorphism $\widetilde{\mathcal{X}} \cong X \times \Delta_d$ as families and let $J_t$ be the complex structure of $\tilde{X}_t$ (the fibre of $\tilde{\mathcal{X}}$ on $t$) such that $\Tilde{X}_t = (X, J_t)$. Take a suitable open neighbourhood $x \in \mathcal{U} \subset \mathcal{X}$ with the following properties: 
\begin{itemize}
    \item The central fiber $U_0$ of the family $\mathcal{U} \to \Delta$ admits a uniformization 
    \begin{equation}
    q:\mathrm{B}(0;1) \to U_0    
    \end{equation}
    on the unit ball $\mathrm{B}(0;1)$ of $\mathbb{C}^2$ such that the standard coordinate $(z,w)$ is a normal coordinate with respect to $q^*g_0$, i.e. $q^*g_0 = g^\mathrm{Euc} + O(r^2)$ under this coordinate,
    
    \item $\mathcal{U} \to \Delta$ induces a smooth map $\zeta: \Delta_d \to \mathfrak{h}_\mathbb{R} \oplus \mathfrak{h}_\mathbb{C}$ as we described above (cf. (\ref{eq:period}) and Definition \ref{dfn:non-deg}). 
    
    \item The simultaneous resolution $\widetilde{\mathcal{X}} \to \mathcal{X}_d$ is induced by an embedding $\iota : \mathcal{U}_d \to \mathcal{Y}_{(0,\zeta_c)}$ and the simultaneous resolution $\mathcal{Y}_{(\zeta_0, \zeta_c)} \to \mathcal{Y}_{(0,\zeta_c)}$ for some $\zeta_0$ which is not perpendicular to any root (cf. (\ref{eq:sim.resol.byALE})).
    
    \item Let $U = \{y \in Y \mid r(y) <1 \}$ be the pull-back of the unit disc $\Delta^2 /\Gamma \subset \mathbb{C}^2/\Gamma$ by the resolution $Y \to \mathbb{C}^2/\Gamma$. Note that $U$ is naturally identified with an open subset of $X$. 
    Then, the trivialization $\widetilde{\mathcal{X}} \cong X \times \Delta_d$ induces 
    \begin{equation}
        \widetilde{\mathcal{U}} \cong U \times \Delta_d
    \end{equation}
    as $C^\infty$ family.
    Recall that $r$ is the radius function on $Y$ (see section \ref{sec:ALE}).

\end{itemize}

In the following, we identify $\mathcal{U}_d$ as an open subset of $\mathcal{Y}_{(0,\zeta_c)}$, $U \subset X$ as an open subset of $Y$ and we fix the trivializations of $\widetilde{\mathcal{X}}$ and $\widetilde{\mathcal{U}}$. Furthermore, by suitable rescaling or taking $\mathcal{U}$ sufficiently small if necessary, we assume that $U \subset Y$ satisfies that $U = \{y \in Y \mid r(y) <1\}$.

Under the above preparations, we construct a family of $J_t$-Hermitian metrics $h_t$ on $X_t = (X, J_t)$ parameterized by $t \in (0, 1)$. For simplicity, we restrict the family along a real ray in $\Delta_d$ through the origin (hence identify a real number $t$ with a complex number $z = te^{i\theta} \in \Delta_d$ with a fixed argument $\theta$). We introduce some parameters to construct metrics:
\begin{equation}\label{def:b}
    b=\varepsilon^\beta =t^{p\beta/2}
\end{equation}
where $\beta> \frac{1}{2}$ is any real number sufficiently close to $1/2$, and $p$ is the order of vanishing of $\zeta$ at $0 \in \Delta_d$ (see (\ref{eq:order of zeta})).
Following Biquard--Rollin \cite{BR}, we construct a family of Riemannian metrics $\Tilde{h}_t$ on X by
\begin{equation}\label{def:h-tilde}
    \Tilde{h}_t=\left\{
        \begin{array}{cc}
           g_0  & 4b \leq r,\  \text{or on}\ X\backslash U \\
           dd^c_0 (r^2 + \chi_{(2b,4b)}(r)\eta)  & 2b \leq r \leq 4b \\
           dd^c_0 r^2 +(1-\chi_{(b, 2b)}(r))\varepsilon^2 H^*_{(\varepsilon^{-1}, \zeta(t))}\xi_t  & b \leq r \leq 2b \\
           g_{\zeta(t)}  & r\leq b
        \end{array}
        \right.
\end{equation}
where $g_0 = dd^c_0(r^2 + \eta)$ on the uniformization $\mathrm{B}(0;1)$, $g_{\varepsilon^{-2}\zeta(t)} = (Euc.) + \xi_t$ as $r \to \infty$ on $Y_{\zeta/\varepsilon^2}$, $d_t^c$ is the $d^c$ operator with respect to $J_t$ and $\chi_{x_1,x_2}(x): \mathbb{R} \to \mathbb{R}$ is a cut-off function defined by 
\begin{equation}\label{def:bump}
    \chi_{x_1,x_2}(x) = \chi\left(\frac{x-x_1}{x_2-x_1}\right)
\end{equation}
for $\chi$, a non-deceasing and non-negative smooth function identically 0 on $x \leq 0$ and 1 on $1 \leq x$. 
Then the Hermitian metric $h_t$ on $T\tilde{X}_t$ is defined by
\begin{equation}\label{def:h}
h_t(u,v) = \frac{\Tilde{h}_t(J_t u, J_t v) + \Tilde{h}_t(u,v)}{2}. 
\end{equation}
They are not \Kahler\ in general. However, by Biquard-Rollin\cite{BR}, we can find a \Kahler\ metric $g_t$ in $\kappa_t$ sufficiently close to $h_t$.
\begin{dfn}[\cite{BR} Section 3.3]
     A weighted H\"{o}lder norm $\|\cdot\|_{C^{k,\alpha}_\delta (X,h_t)}$ on $X$ with respect to $h_t$ is defined as follows:\\
 Takw a family of weight functions $\rho_t$ on $X$ so that
    \begin{equation}
        \rho_t = \left\{
        \begin{array}{cc}
         \text{positive and  smooth}    &  1\leq r \\
         r    &  \varepsilon \leq r \leq 1 \\
          \varepsilon    & r \leq \varepsilon
        \end{array}
        \right. .
    \end{equation}
    Using it, the weighted $C^{k,\alpha}_{\delta}$-H\"{o}lder norm on $(X, h_t)$ is defined as
    \begin{equation}
        \|f\|_{C^{k,\alpha}_\delta (X,h_t)} := \sum_{j=0}^k \sup |\rho_t^{j-\delta} \nabla^j f|_{h_t} + \|\rho_t^{k+\alpha-\delta} \nabla^k f\|_\alpha,
    \end{equation}
    where $\|\cdot\|_\alpha$ is the H\"{o}lder coefficient with exponent $\alpha$, i.e.
    \begin{equation*}
        \|\phi \|_\alpha = \sup_{x \neq y} \frac{\|\phi(x) - \phi(y) \|_{h_t}}{d_{h_t}(x,y)^\alpha}.
    \end{equation*}
    Note that these norms are independent of the choice of $\rho_t$ up to equivalence of norms.
\end{dfn}

\begin{prop}[\cite{BR} Corollary 27. Proposition 30.]\label{prop:BR}
    For the family of Hermitian metrics $\{h_t\}_{0 < t < 1}$ constructed above, there exists a family of \Kahler\ metrics $\{g_t\}_{0< t <1 }$ with \Kahler\ classes $\{\kappa_t\}_{0<t <1}$ such that 
    \begin{equation}
        \|h_t - g_t \|_{C^{2,\alpha}_{\delta,t}}= O(\varepsilon^2),
    \end{equation}
    for any $\delta \in (-2,0)$ sufficiently close to $-2$.
\end{prop}

The following a priori estimate of Biquard-Rollin plays an essential role in this paper:

\begin{thm}[\cite{BR} Corollary 35]\label{thm:BR}
    Let $(\mathcal{X}, \{\kappa_t\})$ be a non-degenerate family with discrete $Aut(X_0)$. For simplicity, we assume that $X_0$ admits only one singularity $x_0$. Let $\{g_t\}$ be the family of \Kahler\ metrics as in Proposition 1. Then for any $t\neq 0$, there exists a unique cscK metric $g_t + dd_t^c \phi_t$ on $\Tilde{X}_t$ ($t \neq 0$) with the following estimate:
    \begin{equation}
        \| \phi_t \|_{C^{4,\alpha}_{\delta+2, t}} \leq C \varepsilon^{2-\beta(\delta +2)}
    \end{equation}
    where $\delta \in (-2,0)$ is any number sufficiently close to $-2$, $\beta= \frac{2}{2-\delta}> \frac{1}{2}$ and $C$ is a constant independent of $t$.
\end{thm}
\section{Main results}\label{sec:Main}
 In this section, we study behavior of the second Chern forms of cscK metrics on a family of complex surfaces. In particular, we prove its regularity as a current on the relative tangent bundle for a non-degenerate smoothing of an orbifold cscK metric with ADE singularity. Namely, for a smoothing family $(\mathcal{X}, \{\kappa_t\}) \to \Delta$ of a complex surface $(X_0, g_0)$, we consider a function $F(t) := \int_{X_t} f|_{X_t} c_2(g_t)$ on $\Delta$. Then, our main result is the following theorem.
\begin{thm}\label{thn:main}
Let $(\mathcal{X} \to \Delta, \kappa_t)$ be a smoothing holomorphic proper family of compact orbifold cscK surface $(X_0, \omega_0)$ with ADE-singularities and $\mathrm{Aut}(X_0)$ discrete, and $f$ be a smooth function on $\mathcal{X}$ with compact support around $x_0\in X_0 \cong 0 \in \mathbb{C}^2/\Gamma$. Assume that the family admits a simultaneous resolution $\widetilde{\mathcal{X}} \to \Delta_d$ after taking a base change by the ramified covering $\Delta_d \to \Delta$ of degree $d$ and it is non-degenerate at $x_0$ along a real ray $\Delta\cap (0,1)$. Consider a function 
\begin{equation}
F(t) = \int_{X_t} f_t c_2(\hat{g}_t) 
\end{equation}
on $\Delta \cap (0,1)$, where $\hat{g}_t$ is the cscK metric on $X_t$ with its \Kahler\ class $\kappa_t$, $c_2(\hat{g}_t)$ is the second Chern form of $\hat{g}_t$ on $TX_t$ and $f_t = f |_{X_t}$. Then $F$ extends to a H\"{o}lder continuous function on $\Delta\cap [0,1)$ with H\"older exponent at least $\frac{1}{d}$.

\end{thm}

We use the Biquard-Rollin's construction to describe $\hat{g}_t$. Recall that we have families of Hermitian metrics $h_t$, \Kahler\ metrics $g_t$, and cscK metrics $\hat{\omega}_t=\omega_t + dd^c_t \phi_t$ with the following estimates along a real ray in $\Delta_d$ (see proposition \ref{prop:BR} and theorem \ref{thm:BR} in section \ref{sec:BR}):
\begin{align}
    \|h_t- g_t\|_{C^{2,\alpha}_{\delta,t}} &= O(\varepsilon^2)\\
    \|\phi_t\|_{C^{4,\alpha}_{\delta+2,t}}&= O(\varepsilon^{2-\beta(\delta+2)}).
\end{align}

 As $h_t$ and $\hat{g}_t$ are hermitian metrics on the same holomorphic vector bundle $TX_t$, we have the following equality by the Bott-Chern formula:
\begin{equation}\label{eq:BottChern}
    c_2(\hat{g}_t) = c_2(h_t) +dd^c_t \tau,
\end{equation}
where $\tau$ is an explicitly given $(1,1)$ form (whose explicit description will be given later, see (\ref{def:tau})). By (\ref{eq:BottChern}) and the Stokes theorem, we can express $F$ as 

\begin{equation}\label{eq:decom.of F}
    F(t) = \int_{X_t} f_t c_2(\hat{g}_t) = \int_{X_t} f_t c_2(h_t) + \int_{X_t} dd^c_t f_t\wedge \tau .
\end{equation}
Then, in the rest of the proof, we consider the two integrals 
\begin{equation}\label{int:h}
    \int_{X_t} f_t c_2(h_t)
\end{equation}
and 
\begin{equation}\label{int:tau}
    \int_{X_t} dd^c_t f_t \wedge \tau.
\end{equation}
For the second integral (\ref{int:tau}), as $f$ is an arbitrary smooth function with compact support, it is sufficient to study an integral
\begin{equation}\label{int:tau'}
    \int_{X_t} \sigma_t \wedge \tau
\end{equation}
for an arbitrary smooth $(1,1)$-form $\sigma$ on $\mathcal{X}$ with compact support (i.e. a compactly supported smooth $2$-form on the total space $\mathcal{X}$ such that its restrictions on fibers $X_t = (X,J_t)$ are $(1,1)$-forms with respect to $J_t$). Here, smooth form on a singular analytic space is by definition a smooth form defined on $\mathcal{X}^{\mathrm{reg}}$ which is a pullback by an embedding into the Euclidean space around a singularity.
As in the construction of $h_t, g_t$ and $\phi_t$, we fix a ramified covering $\Delta_d \to \Delta$, a simultaneous resolution $\widetilde{\mathcal{X}} \to \mathcal{X}$ and an open neighbourhood $\mathcal{U} \ni x_0$ (see section \ref{sec:BR} for precise description). Then, by considering the pull-back of $F$ on $\Delta_d$, the integrals can be expressed as integrals on $X$. We divide the integrals into five parts, integrals on the following five domains: 
\begin{equation*}
    X = (X\backslash U) \cup U_{r \leq b} \cup U_{b\leq r\leq 2b} \cup U_{2b\leq r \leq 4b} \cup U_{4b\leq r \leq 1},
\end{equation*}
where
\begin{equation*}
    U_{r \leq b} = \{ x \in U \mid r(x) \leq b \}
\end{equation*}
and the others are defined similarly.
\begin{center}
\begin{tikzpicture}
\draw (-1.5, -3.7) rectangle (1.5, -2);
\draw (-2.5, -3.8) rectangle (2.5, 0);
\draw (-3, -3.9) rectangle (3, 2.1);
\draw (-4, -4) rectangle (4,4);
\draw (-4.5, -4.5) rectangle (7,4.5);

\draw (5.5,3) node{$X\backslash U$};
\draw (5.5,2.5) node{$\tilde{h}_t := g_0$};
\draw (5.5 ,2) node{cf. Lemma \ref{lem:Outside}};

\draw (0,-2.3) node{$ U_{r\leq b}$};
\draw (0,-2.8) node{$\tilde{h}_t: = g_{\zeta(t)}$};
\draw (0,-3.3) node{cf. Lemma \ref{lem:b}};

\draw (0,-0.3) node{$U_{b \leq r \leq 2b}$};
\draw (0,-0.7) node{$\tilde{h}_t := dd^c_0 r^2 +$};
\draw (0,-1.2) node{$(1-\chi_{(b, 2b)}(r))\varepsilon^2 H^*_{(\varepsilon^{-1}, \zeta(t))}\xi_t$};
\draw (0, -1.6) node{cf. Lemma \ref{lem:b2b}};

\draw (0,1.5) node{$U_{2b \leq r \leq 4b}$};
\draw (0, 1) node{$\tilde{h}_t := dd^c_0 (r^2 + \chi_{(2b,4b)}(r)\eta)$};
\draw (0,0.5) node{cf. Lemma \ref{lem:2b4b}};

\draw (0,3.5) node{$ U_{4b \leq r \leq 1}$};
\draw (0, 3.0) node{$\tilde{h}_t := g_0$};
\draw (0,2.5) node{cf. Lemma \ref{lem:4b1}};

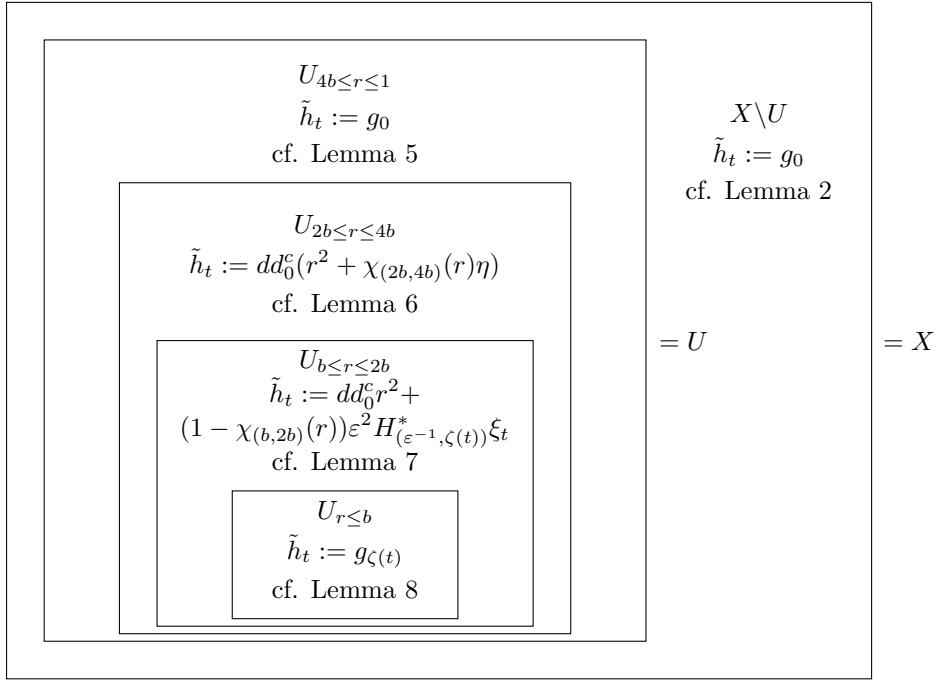
\draw (0, -5) node{Division of domains.};
\draw (7.5, 0) node{$= X$};
\draw (4.5, 0) node{$=U$};

\end{tikzpicture}
\end{center}
Recall that $U$ is diffeomorphic to the minimal resolution of a suitable open neighbourhood of $x_0 \in X_0$ with $r < 1$ (see section \ref{sec:BR}). Let $\tilde{f}$ be the pull-back of $f$ on $\widetilde{\mathcal{X}}$ and $\tilde{f}_t$ be its restriction to $X \times \{t\}$.

\begin{lem}\label{lem:Outside}
    $\int_{X\backslash U} \tilde{f}_t c_2(h_t) = \int_{X\backslash U} \tilde{f}_0 c_2(g_0)+O(t^d)$.
\end{lem}
\begin{proof}
Recall that $\tilde{h}_t = g_0 $ on $X\backslash U$. As $\|J_0 - J_t\| = O(t^d)$ on the domain, we have $h_t = g_0 +O(t^d)$. Then, $c_2(h_t) = c_2(g_0) + O(t^d)$. As $f$ is a smooth function on $X \times \Delta_d$, we have an expansion $\tilde{f}_t(z) = \tilde{f}_0(z) + \dot{f}_0(z) t^d + O(t^{2d})$ by Taylor's theorem with respect to $t$ (note that the family $\mathcal{X}\backslash\mathcal{U} \to \Delta$ admits a local trivialization near the origin $t=0$ as it is a family of smooth manifolds with boundaries). Then we have
\begin{equation}
\begin{aligned}
    \int_{X\backslash U} \tilde{f}_t c_2(h_t) &= \int_{X\backslash U} (\tilde{f}_0 + \dot{f_0} t^d + O(t^{2d}) )(c_2(g_0) + O(t^d))\\
    &= \int_{X\backslash U} \tilde{f}_0 c_2(g_0)+O(t^d).
    \end{aligned}
\end{equation}
\end{proof}

Recall that $\mathcal{U}_d \to \Delta_d $ can be embedded into a family $\mathcal{Y}_{(0,\zeta_c)} \to \Delta_d$ whose fiber on $t \in \Delta_d$ is $Y_{(0,\zeta_c(t))}$ by the construction of $\mathcal{U}$. In particular, $U_t = (U,J_t)$ can be identified with an open subset of $Y_{(0,\zeta_c(t))} = (Y, I_{(0,\zeta_c(t))})$ such that $U = \{ y \in Y \mid r < 1\}$ for $t \neq 0$ . Under this identification, regard the local frame $\{e^1_{\zeta(t)}, e^2_{\zeta(t)}\}$ of Lemma \ref{lem:cpx.str} to be vector fields defined on $U$ (parameterized by $t$). In the rest of this section, we fix this identification. Also, note that we have a trivialization of $\mathcal{Y}_{\zeta} = \mathcal{Y}_{(\zeta_0, \zeta_c)}$ (hence, that of $\widetilde{\mathcal{U}}$) in (\ref{eq:triv.byALE}), where $\zeta_0 \in \mathfrak{h}_\mathbb{R}$ is any fixed vector which is not perpendicular to any root. 

For an analysis of $\int_{U_t} \tilde{f}_t c_2(h_t)$, it is sufficient to study its pull-back by $H_{(\varepsilon, \zeta(t)/\varepsilon^2)}$.
Namely, we investigate $\int_{Y_{r \leq 1/\varepsilon}} H^*_{(\varepsilon, \zeta(t)/\varepsilon^2)}(\tilde{f}_t) H^*_{(\varepsilon, \zeta(t)/\varepsilon^2)}(c_2(h_t))$ in the following. For $H^*_{(\varepsilon, \zeta(t)/ \varepsilon^2)}(\tilde{f}_t)$, we have the following lemma.
\begin{lem}\label{lem:lemma 3}
    $H^*_{(\varepsilon, \zeta(t)/\varepsilon^2)}(\tilde{f}_t) = f(x_0) + O(\varepsilon^{k} r^k)$ on $Y_{r \leq 1/\varepsilon}$ as $t \to 0$, where $k$ is an integer determined by the singularity and is greater than or equal to 2.
\end{lem}
\begin{proof}
    Let $\mathcal{Y}_{(0,\zeta_c)} \to \Delta_d$ be the analytic family defined by $\zeta_c : \Delta_d \to \mathfrak{h}_\mathbb{C}$ (see section \ref{sec:ALE}). Then, by the general theory of deformations of ADE singularities, we have an embedding $\iota_\zeta : \mathcal{Y}_\zeta \to \mathbb{C}^3 \times \Delta_d$ induced by the embedding $\iota : \mathcal{Y} \to \mathbb{C}^3 \times \mathfrak{h}_\mathbb{C}$ (see also (\ref{eq:embedding})). Assume that $\iota_\zeta$ maps the singularity $x_0 \in Y_{(0,0)} \cong 0 \in \mathbb{C}^2/\Gamma$ to the origin: $\iota_\zeta (x_0) = (0,0)$. This embedding induces the following commutative diagram:

\begin{equation}\label{diag:embedding of U}
     \begin{CD}
 \mathcal{U}_d @>>> \mathbb{C}^3 \times \Delta_d\\
@VVV @VVV\\
     \mathcal{U}@>>> \mathbb{C}^3 \times \Delta,
     \end{CD}
\end{equation}
where $\mathbb{C}^3 \times \Delta_d \to \mathbb{C}^3 \times \Delta$ is given by $(x,t) \mapsto (x,t^d)$. As $f$ is a smooth function on $\mathcal{U}$, we have an extension $\mathfrak{f}$ of $f$ on $\mathbb{C}^3 \times \Delta$. Let $\mathfrak{f} = F_0(x) + F_1(x)t + O(t^2)$ be the Taylor expansion with respect to $t$ along a real ray $\Delta \cap \mathbb{R}$. Then, by the commutative diagram (\ref{diag:embedding of U}), the pull-back of $f$ on $\mathcal{U}_d |_{\Delta_d \cap (0,1)}$ is given by the restriction of
\begin{equation}\label{exp:F}
    F_0(x) + F_1(x)t^d + O(t^{2d})
\end{equation}
  on $\mathcal{U}_d |_{\Delta_d \cap (0,1)}$. We will show that $|\iota_\zeta (y) | = O(r^k)$ for sufficiently small $t$, where $k = \max\{k_1,k_2,k_3\} \geq 2$ is the weight of the weighted $\mathbb{C}^*$ action on $\mathbb{C}^3$ (see section 2.1 and \cite{Slo1} section 2.5 for details and their explicit values). In fact, $|\iota_\zeta (y)| = O(r^k)$ for $y \in \mathcal{Y}_{\zeta(0)}$ by its explicit form of embedding. For instance, if $\Gamma$ is the cyclic group of order $n$ (i.e. $A_{n-1}$-type case), $\mathbb{C}^2/\Gamma$ can be embedded into $\mathbb{C}^3$ by 
  \begin{equation*}
  \begin{aligned}
    \mathbb{C}^2 &\to \mathbb{C}^3\\
      (u,v) &\mapsto (u^n, v^n, uv)
  \end{aligned}
  \end{equation*}
  and the weight $k = (k_1, k_2, k_3)$ is $(n,n,2)$, see \cite{Slo1} for $D_n$ and $E_n$ cases. 
  For $y \in Y_{\zeta(t)}, t\neq 0$, if $r(y) = \alpha > 1$, consider $y'= H_{(1/\alpha, \zeta)}(y) \in Y_{\zeta(t)/\alpha^2}$ which satisfies $r(y')=1$. Then, as $\iota$ is $\mathbb{R}_{>0}$ equivariant (cf. (\ref{eq:embedding})), we have $\iota_{\zeta}(y) = \iota_\zeta(H_{(\alpha,\zeta/\alpha^2)}(y'))= \alpha\cdot \iota_{\zeta/\alpha^2}(y')$, where $\iota_{\zeta/\alpha^2} : \mathcal{Y}_{\zeta/\alpha^2} \to \mathbb{C}^3 \times \Delta_d$ is the embedding induced by $\zeta/\alpha^2 :\Delta_d \to \mathfrak{h}_\mathbb{C}$. As $\bigcup_{\alpha > 1}\iota_{\zeta/\alpha^2}(\{r \leq 1\})$ is relatively compact in $\mathbb{C}^3 \times \Delta_d$, $\sup_{\alpha \leq 1, r(y)=1}|\iota_{\zeta/\alpha^2}(y)|$ is bounded. Therefore, we have $|\iota_\zeta(y)| = O(r^k)$. By the Taylor expansion of $F_0$, we have the following estimate for $f^{(d)}$, the pull-back of $f$ on $\mathcal{U}_d$,
\begin{equation}\label{est:f on Ud}
    f^{(d)} = F_0(0) + O(r^k) + O(r^k)t^d + O(t^d).
\end{equation}
In particular, $\tilde{f}$, the pull back of $f^{(d)}$ on $\tilde{\mathcal{U}}$, also satisfies the estimate (\ref{est:f on Ud}). Therefore, we have
\begin{equation}
    H^*_{(\varepsilon,\zeta/\varepsilon^2)}(\tilde{f}) = f(x_0) + O(\varepsilon^k r^k).
\end{equation}
    \end{proof}
    The next lemma is also important for our analysis.
\begin{lem}
    $H_{(\varepsilon, \zeta/\varepsilon^2)}^*(\tilde{f}_t) = H_{(\varepsilon, \zeta/\varepsilon^2)}^*(\tilde{f}_0) + O(t)$ on $Y_{1 < r < 1/\varepsilon}$.
\end{lem}
\begin{proof}
   Consider the composition 
   \begin{equation}
   \Phi: U \times \Delta_d \to \tilde{\mathcal{U}} \to \mathcal{U}_d \to \mathbb{C}^3 \times \Delta_d,
  \end{equation} 
   where the trivialization $U \times \Delta_d \to \tilde{\mathcal{U}}$ is induced by the trivialization $Y \times \Delta_d \cong \mathcal{Y}_{(\zeta_0, \zeta_c)}$ (cf. (\ref{eq:triv.byALE})) and the map $\mathcal{U}_d \to \mathbb{C}^3 \times \Delta_d$ is given by $\iota_\zeta$. Let $\Phi(u,t) = (x(u,t),t)$. Then, by (\ref{exp:F}), $\tilde{f}$ can be represented as follows:
   \begin{equation}
       F_0(x(u,0)) + l(\dot{x}(u,0))t + O(t^2),
   \end{equation}
   where $l$ is a linear form and $\dot{x}(u,0)$ is the derivative with respect to $t$. Note that $F_0(x(u,0)) = \tilde{f}_0$. And, as $x(u,t)$ is a restriction of a smooth function on $Y$ (parameterized by $t$ smoothly) to $U$, $|\dot{x}(u,0)|$ is bounded on $r<1$ i.e. on $U \subset Y$ (hence, its pull back by $H_{(\varepsilon,\zeta/\varepsilon^2)}$ is also bounded on $r < 1/\varepsilon$). Therefore we have the desired estimate.
\end{proof}
By these lemmas, we can derive estimates of the first integral of (\ref{eq:decom.of F}) on $U$.
\begin{lem}\label{lem:4b1}
    $\int_{Y_{4b \leq r \leq 1}} \tilde{f}_t c_2(h_t) = \int_{Y_{4b \leq r \leq 1}} \tilde{f}_0 c_2(g_0) + O(t)$.
\end{lem}
\begin{proof}    
 Recall that $\tilde{h}_t = g_0$ on $U_{4b \leq r \leq 1}$, and then $\tilde{h}_t$ is $I_0$-Hermitian on this domain. As the coordinate $(z_0,w_0)$ is $g_0$-adapted, we have
 \begin{equation*}
     g_0 = g^{\mathrm{Euc}} + O(r^2)
 \end{equation*}
 on $U_{4b \leq r \leq 1}$.
 Hence, $H_{(\varepsilon,\zeta/\varepsilon^2)}^*(g_0) = \varepsilon^2 g^{\mathrm{Euc}} + O(\varepsilon^4 r^2)$ on $Y_{4b/\varepsilon \leq  r \leq 1/\varepsilon}$. Furthermore, as $I_\zeta = I_0 + O(r^{-4})$ on $Y$, we have 
 \begin{equation}
     H_{(\varepsilon,\zeta/\varepsilon^2)}^*(h_t) = H_{(\varepsilon,\zeta/\varepsilon^2)}^*(g_0) + O(\varepsilon^2r^{-4})
 \end{equation}
 on $Y_{4b/\varepsilon \leq r \leq 1/\varepsilon}$. Therefore, we have
\begin{equation}
    (H_{(\varepsilon,\zeta/\varepsilon^2)}^*(h_t))^{-1} = (H_{(\varepsilon,\zeta/\varepsilon^2)}^*(g_0))^{-1} + O(\varepsilon^{-2}r^{-4}),
    \end{equation}
    
    \begin{equation}
    \begin{aligned}
    \partial_t (H_{(\varepsilon,\zeta/\varepsilon^2)}^*(h_t)) &= (1 + O(r^{-4}))\partial_0(H_{(\varepsilon,\zeta/\varepsilon^2)}^*(g_0) +O(\varepsilon^2 r^{-4})) \\
    &= \partial_0 (H_{(\varepsilon,\zeta/\varepsilon^2)}^*(g_0)) + O(\varepsilon^2 r^{-5}),
    \end{aligned}
    \end{equation}
    and
    \begin{equation}
    \begin{aligned}
        \partial^2_t (H_{(\varepsilon,\zeta/\varepsilon^2)}^*(h_t)) &= ((1 + O(r^{-4}))\partial_0)^2(H_{(\varepsilon,\zeta/\varepsilon^2)}^*(g_0) +O(\varepsilon^2r^{-4}))\\
        &=\partial^2_0 (H_{(\varepsilon,\zeta/\varepsilon^2)}^*(g_0)) + O(\varepsilon^2 r^{-6}),
    \end{aligned}
    \end{equation}
where $\partial_t$ is any differentiation with respect to the frame $e^i_{\zeta(t)}$ on $Y_{4b/ \varepsilon \leq r \leq 1/\varepsilon}$. 
Hence we have 
\begin{equation}\label{eq:curve.on4b1}
    R^i_{jk\overline{l}}(H_{(\varepsilon,\zeta/\varepsilon^2)}^*(h_t)) = R^i_{jk\overline{l}}(H_{(\varepsilon,\zeta/\varepsilon^2)}^*(g_0)) + O(r^{-6}),
\end{equation}
\begin{equation}
    \begin{aligned}
        c_2(H_{(\varepsilon,\zeta/\varepsilon^2)}^*(h_t)) = c_2(H_{(\varepsilon,\zeta/\varepsilon^2)}^*(g_0)) + O(\varepsilon^2r^{-6}) \mathrm{Vol}_{\mathrm{Euc}}
    \end{aligned}
\end{equation}
on $Y_{4b/\varepsilon \leq r \leq 1/\varepsilon}$ under the trivialization by the frame $\{e^1_{\zeta(t)}, e^2_{\zeta(t)}\}$. Therefore, 
\begin{equation}
    \begin{aligned}
        \int_{4b \leq r \leq 1} \tilde{f}_t c_2(h_t) &= 
        \int_{4b/\varepsilon \leq r \leq 1/\varepsilon} H_{(\varepsilon, \zeta/\varepsilon^2)}^*(\tilde{f}_t) c_2(H_{(\varepsilon,\zeta/\varepsilon^2)}^*(h_t)) \\
        &= 
        \int_{4b/\varepsilon \leq r \leq 1/\varepsilon} H_{(\varepsilon,\zeta/\varepsilon^2)}^* (\tilde{f}_t) (c_2(H_{(\varepsilon,\zeta/\varepsilon^2)}^*(g_0)) + O(\varepsilon^2 r^{-6})\mathrm{Vol}_{\mathrm{Euc}}) \\
        &= \int_{4b/\varepsilon \leq r \leq 1/\varepsilon} (H_{(\varepsilon,\zeta/\varepsilon^2)}^*(\tilde{f}_0) + O(t))(c_2(H_{(\varepsilon,\zeta/\varepsilon^2)}^*(g_0)) + O(\varepsilon^2 r^{-6}) \mathrm{Vol}_{\mathrm{Euc}})\\
        &= \int_{4b/ \varepsilon \leq r \leq 1/\varepsilon}H_{(\varepsilon,\zeta/\varepsilon^2)}^*(\tilde{f}_0)c_2(H_{(\varepsilon,\zeta/\varepsilon^2)}^*(g_0)) + O(t) \\
        &= \int_{4b \leq r \leq 1} \tilde{f}_0c_2(g_0) + O (t).
    \end{aligned}
\end{equation}
Note that $g_0$ is an orbifold metric, hence the curvature of $g_0$ is uniformly bounded on this domain. 
\end{proof}

\begin{lem}\label{lem:2b4b}
    $\int_{U_{2b \leq r \leq 4b}} \tilde{f}_t c_2 (h_t) = \int_{U_{2b \leq r \leq 4b}} \tilde{f}_0 c_2(g_0) + O(\varepsilon^2).$
\end{lem}
\begin{proof}
    Recall that $\tilde{h}_t = dd^c_0(r^2 + \chi_{(2b,4b)}(r) \eta)$ on this domain where $g_0 = dd^c_0 (r^2 + \eta )$ with respect to the coordinate $(z_0,w_0)$. Note that $\eta = O(r^4) $ on $U_{2b \leq r \leq 4b}$. Then, we have 
\begin{equation}
    \begin{aligned}
        dd^c_0(r^2 + \chi_{(2b,4b)} (r) \eta)
        &= dd^c_0 r^2 + \frac{ \chi'_{(2b,4b)}(r) \eta }{2b} dd^c_0r + \frac{\chi''_{(2b,4b)}(r) \eta}{4b^2} dr \wedge d^c_0 r\\
        &+ \frac{\chi'_{(2b,4b)}(r)}{2b}( d\eta \wedge d^c_0r + dr \wedge d_0^c \eta) + \chi_{(2b,4b)}(r) dd^c_0\eta.
    \end{aligned}
\end{equation}
Recall that $\eta = O(r^4)$, $r = O(r)= O(b)$, $\|dr\|_{g_0}, \|d_0^cr\|_{g_0}=O(r^0), \|dd_0^cr\|_{g_0} = O(r^{-1})$ and $\chi', \chi''$ are bounded, we have 
\begin{equation}\label{eq:estimate for h}
     dd^c_0(r^2 + \chi_{(2b,4b)} (r) \eta)= g_0 + O(r^2).
\end{equation}
Therefore, 
\begin{equation}
\begin{aligned}
    H_{(\varepsilon, \zeta/\varepsilon^2)}^*(h_t) &= H_{(\varepsilon, \zeta/\varepsilon^2)}^*(\tilde{h}_t) + O(\varepsilon^2r^{-4}) \\
    &= H_{(\varepsilon, \zeta/\varepsilon^2)}^*(g_0) + O(\varepsilon^4r^2) + O(\varepsilon^2r^{-4}) \\
    &= H_{(\varepsilon, \zeta/\varepsilon^2)}^*(g_0) + O(\varepsilon^4r^2)
    \end{aligned}
\end{equation}
on $Y_{2b/ \varepsilon \leq r \leq 4b/\varepsilon}$. In particular, $h_t = H^*_{(\varepsilon^{-1},\zeta)} H_{(\varepsilon, \zeta/\varepsilon^2)}^*(h_t) = g_0 + O(r^2)$ on $U_{2b \leq r \leq 4b}$.
Then, we have
\begin{align}
     (h_t)^{-1} &= (g_0)^{-1} + O(r^2)\\
    \partial_t (h_t) &= (\partial_0 + O(r^2) )(g_0 +O(r^2)) = \partial_0 (g_0) + O(r)\\
    \partial^2_t (h_t) &= (\partial_0 + O(r^2) )^2(g_0 +O(r^2)) =\partial^2_0 (g_0) + O(r^0)
\end{align}
on $U_{2b \leq r \leq 4b}$, where $\partial_t$ is any differentiation by the frame $\{\varepsilon^{-1} (H_{(\varepsilon, \zeta/\varepsilon^2)})_* e_t^i\}$. Then we have the following on $U_{2b \leq r \leq 4b}$:
\begin{equation}\label{eq:curv.on2b4b}
    R^i_{jk\overline{l}}(h_t) = R^i_{jk\overline{l}}(g_0) + O(r^0),
\end{equation}
\begin{equation}
    \begin{aligned}
        c_2(h_t) = c_2(g_0) + O(r^0)\mathrm{Vol}_\mathrm{Euc}
    \end{aligned}
\end{equation}
with respect to the frame $\{\varepsilon^{-1}(H_{(\varepsilon, \zeta/\varepsilon^2)})_* e_t^i \}$.
As the volume of the domain $U_{b\leq r \leq 2b}$ is of order $b^4$ and $c_2(g_0)$ is bounded (note that $g_0$ is an orbifold metric) we have the following estimate:
    \begin{align}
        \int_{U_{2b \leq r \leq 4b}} \tilde{f}_t c_2 (h_t) = 
        \int_{U_{2b \leq r \leq 4b}} \tilde{f}_0 c_2(g_0) + O(b^4).
    \end{align}
    As $b = \varepsilon^\beta$ for $\beta > \frac{1}{2}$, we see that $O(b^4)$ can be replaced by $O(\varepsilon^2)$. 
\end{proof}

\begin{lem}\label{lem:b2b}
   $ \int_{U_{b \leq r \leq 2b}} \tilde{f}_t c_2 (h_t) = \int_{U_{b \leq r \leq 2b}} \tilde{f}_0 c_2(g_0) + O(\varepsilon^2)$.
\end{lem}
   \begin{proof}
        Recall that $\tilde{h}_t = dd^c_0 r^2 + (1-\chi_{(b,2b)}(r)) \varepsilon^2 H^*_{(\varepsilon^{-1},\zeta)}\xi_t$ where $\xi_t = g_{\varepsilon^{-2}\zeta(t)} - g^\mathrm{Euc}$ which satisfies $\xi_t = O(r^{-4})$ on $Y_{b/\varepsilon \leq r \leq 2b/\varepsilon}$, as $g_{\varepsilon^{-2}\zeta(t)}$ is an ALE metric. The estimates $\varepsilon^2 H_{(\varepsilon^{-1},\zeta)}^*(\xi_t) = O(\varepsilon^4r^{-4})$ and $\chi_{(b,2b)} = O(r^0)$ on $U_{b \leq r \leq 2b}$ imply that
    \begin{equation}
        \tilde{h}_t = dd^c_0 r^2 + O(\varepsilon^4 r^{-4}) = g_0 + O(r^2).
    \end{equation}
    Note that this is the same estimate as the case on $U_{2b \leq r \leq 4b}$, hence we obtain
    \begin{equation}
         \int_{U_{b \leq r \leq 2b}} \tilde{f}_t c_2 (h_t) = \int_{U_{b \leq r \leq 2b}} f_0 c_2(g_0) + O(\varepsilon^2).
    \end{equation}
\end{proof}

\begin{lem}\label{lem:b}
     $\int_{U_{r\leq b}} \tilde{f}_t c_2(g_{\zeta(t)}) = f(x_0)\int_{Y_{\dot{\zeta}}} c_2(g_{\dot{\zeta}}) + O(\varepsilon^2)$.
\end{lem}
\begin{proof}
     Recall that $\tilde{h}_t = h_t = g_{\zeta(t)}$ on this domain. By the naturality of Chern forms, we have 
    \begin{equation}
        \int_{U_{r \leq b}} \tilde{f}_t c_2(g_{\zeta(t)}) = \int_{Y_{r \leq b/\varepsilon}} (H_{(\varepsilon, \zeta/\varepsilon^2)}^*\tilde{f}_t) c_2 (g_{\varepsilon^{-2} \zeta(t)}).
    \end{equation}  
    By lemma 3, we have 
    \begin{equation}
        \int_{Y_{r\leq b/\varepsilon} } (H_{(\varepsilon,\zeta/\varepsilon^2)}^* \tilde{f}_t) c_2(g_{\varepsilon^{-2} \zeta(t)}) 
        = f(x_0) \int_{Y_{r\leq b/\varepsilon}} c_2(g_{\varepsilon^{-2} \zeta(t)} )+ 
        \int_{Y_{r\leq b/\varepsilon}} O(\varepsilon^kr^k) c_2(g_{\varepsilon^{-2}\zeta(t)}) .
    \end{equation}
    As $g_{\varepsilon^{-2}\zeta(t)}$ is an ALE metric, we have $g_{\varepsilon^{-2}\zeta(t)} = g^{\mathrm{Euc}} + O(r^{-4})$ as $r \to \infty$. Hence $c_2(g_{\varepsilon^{-2}\zeta(t)}) = O(r^{-12}) \mathrm{Vol}_\mathrm{Euc}$ as $r \to \infty$. Then, 
    \begin{equation}
    \begin{aligned}
        \int_{Y_{r \leq b/\varepsilon}} c_2(g_{\varepsilon^{-2}\zeta(t)}) &= \int_{Y_{\varepsilon^{-2}\zeta(t)}} c_2(g_{\varepsilon^{-2} \zeta(t)}) + O(\varepsilon^8/b^8) \\
         \int_{Y_{r \leq b/\varepsilon}}O(\varepsilon^kr^k) c_2(g_{\varepsilon^{-2}\zeta(t)}) &= O(\varepsilon^2).
        \end{aligned}
    \end{equation}
    Recall that the integral $\int_{Y_{\varepsilon^{-2}\zeta(t)}} c_2(g_{\varepsilon^{-2}\zeta(t)}) = e_{orb}(Y_{\dot{\zeta}})$, the Euler number, is independent of $t$, hence we have 
    \begin{equation}
        \int_{U_{r\leq b}} \tilde{f}_t c_2(g_{\zeta(t)}) = f(x_0)\int_{Y_{\dot{\zeta}}} c_2(g_{\dot{\zeta}}) + O(\varepsilon^2).
    \end{equation}
    \end{proof}
    
    Summing them up, we obtain the following estimate:
    \begin{equation}\label{eq:estimate for c2(h)}
        \int_{X_t} \tilde{f}_t c_2(h_t) = \int_{X_0} \tilde{f}_0 c_2(g_0) + f(x_0)e_{orb}(Y_{\dot{\zeta}}) + O(\varepsilon^2).
    \end{equation}

    We calculate the second term 
    \begin{equation*}
        \int_{X_t} \sigma_t \wedge \tau
    \end{equation*}
    of (\ref{eq:decom.of F}) in the following. Recall that $\sigma$ is initially $dd^c f$, however, we consider a general $2$-form $\sigma$ with bidegree $(1,1)$ on each $X_t$ in the following.
    Before we calculate the integral $\int_{X_t} \sigma_t \wedge \tau $, we recall the Bott-Chern formula to derive an explicit form of $\tau$. Let $\nabla_t, \nabla_t'$ be the Chern connections of $h_t, g_t$ respectively. Take a family of hermitian metrics $\{k_s\}$ that connects $h_t$ and $g_t$ smoothly, say $k_s = h_t + s(g_t - h_t)$ here, parameterised by $s \in [0,1]$ and define $N_s$ by $(N_s)_{i,j} = (\dot{k}_s)_{i\alpha} (k_s)^{\alpha j}$ under a coordinate (the dot means a derivative with respect to $s$). Then Bott-Chern formula (see \cite{BC}) yields 
    \begin{equation*}
        c(h_t) - c(g_t) = dd^c \phi
    \end{equation*}
    for the total Chern forms, where $\phi$ is defined by 
    \begin{equation*}
        \phi = -\frac{1}{2\pi} \int_0^1 N_s \wedge \left( I - \frac{R_s}{2\pi\sqrt{-1}}\right) ds,
    \end{equation*}
     where $R_s$ is the curvature tensor of $k_s$. 
    Under a trivialization, the right hand side becomes
    \begin{equation*}
    \begin{aligned}
     -\frac{1}{2\pi} \int_0^1 N_s \wedge \left( I - \frac{R_s}{2\pi\sqrt{-1}}\right) ds &=
        -\frac{1}{2\pi} \int_0^1 \left(\det
        \begin{bmatrix}
        (N_s)_{11} & -\frac{(R_s)_{12}}{2\pi\sqrt{-1}} \\
        (N_s)_{21} & 1- \frac{(R_s)_{22}}{2\pi\sqrt{-1}} \\
        \end{bmatrix}
        +
      \det
        \begin{bmatrix}
        1 -\frac{(R_s)_{11}}{2\pi\sqrt{-1}} & (N_s)_{12} \\
        - \frac{(R_s)_{21}}{2\pi\sqrt{-1}} & (N_s)_{22} \\
        \end{bmatrix}
        \right) ds \\
        &= -\frac{1}{2\pi} \int_0^1(N_s)_{11} ds + \frac{1}{4\pi^2 \sqrt{-1}} \int_0^1 (N_s)_{11} (R_s)_{22} - (N_s)_{21} (R_s)_{12}ds \\
        &+ -\frac{1}{2\pi} \int_0^1(N_s)_{22} ds + \frac{1}{4\pi^2 \sqrt{-1}} \int_0^1 (N_s)_{22} (R_s)_{11} - (N_s)_{12} (R_s)_{21}ds
        \end{aligned}
    \end{equation*}
    In particular, $\tau$, the $(1,1)$ component of $\phi$, can be written as 
    \begin{equation}\label{def:tau}
        \tau = \frac{1}{4\pi^2 \sqrt{-1}} \int_0^1 (N_s)_{11} (R_s)_{22} - (N_s)_{21} (R_s)_{12} + (N_s)_{22} (R_s)_{11} - (N_s)_{12} (R_s)_{21}ds.
    \end{equation}
   \begin{lem}\label{lem:sigma}
       $|\sigma_t |_{h_t} < C$ for some constant $C$ independent of $t$.
   \end{lem}
\begin{proof}
    It is sufficient to estimate $|\sigma_t|_{h_t}$ around the singularity. Recall that $\sigma$ is smooth form on $\mathcal{X}$. Hence, by definition of a smooth form on (possibly) singular analytic space, we have an embedding $i : V \hookrightarrow \mathbb{C}^m$ of an open neighbourhood $V \subset \mathcal{X}$ of $x_0$ such that $\sigma$ extends to a smooth form on $\mathbb{C}^m$. Furthermore, the embedding $i : V \hookrightarrow \mathbb{C}^m$ factors through the embedding $\iota_\zeta : \mathcal{Y}_\zeta \to \mathbb{C}^3 \times \Delta_d$ i.e. there exists a regular map $\phi : V' \to \mathbb{C}^m$ defined on a neighbourhood $V' := \iota_\zeta(V) \subset \mathbb{C}^3 \times \Delta$ of the origin such that $i = \phi \circ \iota_\zeta$. Therefore, by considering the pullback of $\sigma$ by $\phi$, we may assume $\sigma$ is defined on $\mathbb{C}^3 \times \Delta_d$.
    Hence we consider an embedding $Y_{\zeta(t)} \hookrightarrow \mathbb{C}^3$ and a smooth form $\sigma_t$ on $\mathbb{C}^3$ (parameterized by $t$ smoothly). 
    For a point $y \in Y_{\zeta(t), r < 1}$, take a $g_{\zeta(t)}$-normal coordinate $(V_y, (y_1, y_2))$ centered at $y$ i.e. $y = (0,0)$ and $g_{\zeta(t)} = (\delta_{i\overline{j}} + O(|y|^2)) dy_i d\overline{y}_j$ under this coordinate. 
    Let
    \begin{equation}
        F = (f_1, f_2, f_3) : V_y \to \mathbb{C}^3
    \end{equation} 
    be the restriction of $\iota_{\zeta(t)}$ on $V_y$. Then, for a two form $\sigma = \sigma_{i\overline{j}} dz_i \wedge d\overline{z}_j$ on $\mathbb{C}^3$, its norm $|\sigma |_{g_{\zeta(t)}}$ at $y$ is equal to 
    \begin{equation}
        \begin{aligned}
            |\sigma|_{g_{\zeta(t)}} &= |F^* \sigma|_{g_{\zeta(t)}} \\
            &= \sum_{i,j,k,l} \sigma_{i\overline{j}} \frac{\partial f_i}{\partial y_k} \overline{\frac{\partial f_j}{\partial y_l}}.
        \end{aligned}
    \end{equation}
    Therefore, as $\sigma_{i\overline{j}}$ are smooth functions on $\mathbb{C}^3$, it is sufficient to estimate the quantity
    \begin{equation}
        \left| \frac{\partial f_i}{\partial y_k} \right|.
    \end{equation}
    Recall that $F = (f_1, f_2, f_3)$ is the restriction of $\iota$, hence we have
    \begin{equation}
        \begin{aligned}
            F_* \frac{\partial}{\partial y_i} &= \sum_{k} \frac{\partial f_k}{ \partial y_i} \frac{\partial}{\partial z_k}, \\
            \sum_k \left| \frac{\partial f_k}{\partial y_i} \right|^2  &=
            \left| (\iota_{\zeta(t)})_* \frac{\partial}{\partial y_i} \right|^2_{g_{\mathbb{C}^3}} \\
            &= \left| \frac{\partial }{\partial y_i} \right|_{\iota_{\zeta(t)}^*g_{\mathbb{C}^3}}.
        \end{aligned}
    \end{equation}
     Let $y' \in Y_{\zeta(t)/\varepsilon^2}$ be the image of $y$ by the dilation $H_{(1/\varepsilon, \zeta(t))} : Y_{\zeta(t)} \to Y_{\zeta(t)/\varepsilon^2}$ i.e. $y' = H_{(1/\varepsilon, \zeta(t))}(y)$. Then, as $H^*_{(\varepsilon, \zeta(t)/\varepsilon^2)} g_{\zeta(t)} = \varepsilon^2 g_{\zeta(t)/\varepsilon^2}$, we see that 
\begin{equation}
\begin{aligned}
     |(H_{(\varepsilon, \zeta(t)/\varepsilon^2)})_* (v) |_{g_{\zeta(t)}} &= |v|_{H^*_{(\varepsilon, \zeta(t)/\varepsilon^2)}g_{\zeta(t)}} \\
     &= |v|_{\varepsilon^2 g_{\zeta(t)/\varepsilon^2}} = \varepsilon,
\end{aligned}
\end{equation}
for any unit vector $v \in T_{y'}Y_{\zeta(t)/\varepsilon^2} $. In particular, for any unit vector $u \in T_y Y_{\zeta(t)}$, there exits a unit vector $v \in T_{y'} Y_{\zeta(t)/\varepsilon^2}$ such that $(H_{(\varepsilon, \zeta(t)/\varepsilon^2)})_* (v/\varepsilon) = u$. Then, for a unit vector $v \in T_{y'} Y_{\zeta(t)/\varepsilon^2}$ such that $(H_{(\varepsilon,\zeta(t))})_* (v/\varepsilon) = \frac{\partial}{\partial y_i}$, we obtain 
\begin{equation}
    \left| \frac{\partial}{\partial y_i} \right|_{\iota^*_{\zeta(t)} g_{\mathbb{C}^3}} = 
    \frac{1}{\varepsilon} | v |_{H^*_{(\varepsilon, \zeta(t)/\varepsilon^2)} \iota^*_{\zeta(t)} g_{\mathbb{C}^3}}.
\end{equation}
    Recall that the embedding $\iota_\zeta : Y_\zeta \to \mathbb{C}^3$ is equivariant i.e. $\varepsilon\cdot \circ \iota_{\zeta/\varepsilon^2} = \iota_{\zeta} \circ H_{(\varepsilon, \zeta/\varepsilon^2)}$, where $\varepsilon\cdot$ is the weighted $\mathbb{C}^*$-action on $\mathbb{C}^3$ (see proof of lemma (\ref{lem:lemma 3})). Then, 
    \begin{equation}
    \begin{aligned}
        H^*_{(\varepsilon, \zeta(t)/\varepsilon^2)} \iota^*_{\zeta(t)} g_{\mathbb{C}^3} &=
        \iota_{\zeta(t)/\varepsilon^2}^* (\varepsilon\cdot)^* g_{\mathbb{C}^3} \\
        &= \left( \sum \varepsilon^{2k_i} dz_i d\overline{z}_i \right)|_{Y_{\zeta(t)/\varepsilon^2}}.
        \end{aligned}
    \end{equation}
    Therefore, if we put $(\iota_{\zeta(t)/\varepsilon^2})_*(v) = \sum v_i \frac{\partial}{\partial z_i}$, it is sufficient to show that $\sum \varepsilon^{2k_i - 2} |v_i|^2$ is uniformly bounded. By the equivariance of $\iota_\zeta$ again, if we put $(\iota_\zeta)_*\frac{\partial}{\partial y_i} = \sum_j (y_i)_j \frac{\partial}{\partial z_j}$, the following holds:
    \begin{equation}
    \begin{aligned}
        \sum_j (y_i)_j \frac{\partial}{\partial z_j} &= \frac{1}{\varepsilon}(\varepsilon\cdot)_* \left( \sum v_j \frac{\partial}{\partial z_j} \right) \\
        &= \sum_j \varepsilon^{k_j-1} v_j \frac{\partial}{\partial z_j}.
        \end{aligned}
    \end{equation}
    Now, consider the equality 
    \begin{equation}\label{eq:equality}
        \sum \varepsilon^{2k_j - 2} |v_j|^2 = \sum |(y_i)_j|^2.
    \end{equation}
    The right hand side of (\ref{eq:equality}) is uniformly bounded as $y$ runs $Y_{\zeta(t), r<1}$ for each (fixed) $t$. On the other hand, since $g_{\zeta(t)/\varepsilon^2} \to g_{\dot{\zeta}}$, there exists a constant $C>0$ independent of $t$ (possibly depends on $r = r(y')$) such that 
    \begin{equation}
        C^{-1} < |v|_{g_{\dot{\zeta}}} < C
    \end{equation}
    for any unit vector $v \in T_{y'}Y_{\zeta(t)/\varepsilon^2}$ with respect to $g_{\zeta(t)/\varepsilon^2}$ for any sufficiently small $t$.
    Then, 
\begin{equation*}
     \{ v \mid v \in T_{y'}Y_{\zeta(t)/\varepsilon^2}, |v|_{g_{\zeta(t)/\varepsilon^2}} =1, r(y') = r\} \subset  \{v\mid v \in T_{y'}Y_{\zeta(t)/\varepsilon^2}, C^{-1} < |v|_{g_{\dot{\zeta}}} <C , r(y') = r\}
\end{equation*}
    for any $t$. In particular, 
    \begin{equation*}
      \bigcup_{t << 1}  \{ (\iota_{\zeta(t)/\varepsilon^2})_* (v) \mid v \in T_{y'}Y_{\zeta(t)/\varepsilon^2}, |v|_{g_{\zeta(t)/\varepsilon^2}} =1, r(y') = r\} 
    \end{equation*}
    is contained in a relatively compact subset 
    \begin{equation*}
        \{(\iota_{\zeta(t)/\varepsilon^2})_* (v) \mid v \in T_{y'}Y_{\zeta(t)/\varepsilon^2}, C^{-1} < |v|_{g_{\dot{\zeta}}} <C , r(y') = r\}
    \end{equation*}
    of $\mathbb{C}^3$ for each $r$ (note that $\iota_{\zeta(t)/\varepsilon^2}$ defines a smooth family of embedding $\{\iota_{\dot{\zeta} + O(\varepsilon)} : Y_{\dot{\zeta} + O(\varepsilon)} \mathbb{C}^3\}_{0 \leq t \leq 1}$, hence an image of a relatively compact subset is a relatively compact). Then the left hand side of (\ref{eq:equality}) is bounded from above uniformly with respect to $t$ (recall that $k_i \geq2$). Then in particular, $\left| \frac{\partial f_i}{\partial y_k} \right|$ is uniformly bounded. Therefore, $|\sigma_t|$ is uniformly bounded as desired.
\end{proof}
   
    \begin{lem}
        $\int_{X_t} \sigma_t\wedge \tau = O(\varepsilon^{2-\beta(\delta+2)})$.
    \end{lem}
    \begin{proof}
        
    By lemma \ref{lem:sigma}, it is sufficient to estimate $\int_{X_t} |\tau|_{h_t} \mathrm{Vol}_{h_t}$.
    By Biquard-Rollin's a priori estimate, we have
    \begin{align}
        \mathrm{sup}| \rho^{j-\delta} \nabla^{j}(g_t - h_t)|_{h_t} = O (\varepsilon^{2-\beta(\delta + 2)})
    \end{align}
    for $j=0,1,2$. More explicitly, 
    \begin{equation}
    |\nabla^j(g_t-h_t)|_{h_t} = \left\{
        \begin{array}{ccc}
              O(\varepsilon^{2-\beta(\delta + 2)}) & 1 \leq r \\
              r^{\delta - j} O(\varepsilon^{2-\beta(\delta + 2)}) & \varepsilon \leq r \leq 1\\
             O(\varepsilon^{(1-\beta)(2+\delta) - j}) & r \leq \varepsilon
        \end{array}\right. .
    \end{equation}
    Therefore, we have 
    \begin{equation}
        |\dot{k}_t|_{h_t} = |(h_t - g_t)|_{h_t}=\left\{ 
         \begin{array}{ccc}
              O(\varepsilon^{2-\beta(\delta + 2)}) & 1 \leq r \\
             r^\delta O(\varepsilon^{2-\beta(\delta + 2)}) & \varepsilon \leq r \leq 1\\
             O(\varepsilon^{(1-\beta)(2+\delta) }) & r \leq \varepsilon
        \end{array}\right.,\\
    \end{equation}
    \begin{equation}
        |k^{-1}_t - h_t^{-1}|_{h_t} = \left\{
        \begin{array}{ccc}
             O(\varepsilon^{2-\beta(\delta + 2)}) & 1 \leq r \\
             r^\delta O(\varepsilon^{2-\beta(\delta + 2)}) & \varepsilon \leq r \leq 1\\
             O(\varepsilon^{(1-\beta)(2+\delta) }) & r \leq \varepsilon
        \end{array}\right..
    \end{equation}
    Hence we obtain the following estimate:
    \begin{equation}
        |N_s|_{h_t} =\left\{ 
         \begin{array}{ccc}
              O(\varepsilon^{2-\beta(\delta + 2)}) & 1 \leq r \\
             r^\delta O(\varepsilon^{2-\beta(\delta + 2)}) & \varepsilon \leq r \leq 1\\
             O(\varepsilon^{(1-\beta)(2+\delta) }) & r \leq \varepsilon
        \end{array}\right..
    \end{equation}  
    For the curvature $R_s$, as we have 
    \begin{equation}
      |\nabla^j (k_s) - \nabla^j(h_t)|_{h_t} =  \left\{
        \begin{array}{ccc}
             O(\varepsilon^{2-\beta(\delta + 2)}) & 1 \leq r \\
             r^{\delta - j} O(\varepsilon^{2-\beta(\delta + 2)}) & \varepsilon \leq r \leq 1\\
             O(\varepsilon^{(1-\beta)(2+\delta) - j}) & r \leq \varepsilon
        \end{array}\right. ,
    \end{equation}
    we have the following:
    \begin{equation}
        |(R_s) - R(h_t)|_{h_t} = \left\{
        \begin{array}{ccc}
             O(\varepsilon^{2-\beta(\delta + 2)}) & 1 \leq r \\
             r^{\delta - 2} O(\varepsilon^{2-\beta(\delta + 2)}) & \varepsilon \leq r \leq 1\\
             O(\varepsilon^{\delta -\beta(\delta+2)}) & r \leq \varepsilon
        \end{array}\right..
    \end{equation}
    By lemma \ref{lem:Outside}, (\ref{eq:curve.on4b1}), and (\ref{eq:curv.on2b4b}), we have the following:
    \begin{equation}
       |R(h_t) - R(g_0)|_{h_t} = \left\{
        \begin{array}{cccc}
          O(t^d) &  1 \leq r \\
          O(\varepsilon^4 r^{-6})  & 4b \leq r \leq 1 \\
          O(r^0)   &b \leq r \leq 4b \\
        \end{array} \right. .
    \end{equation}
    Note that there exists a constant $C >0$ independent of $t$ such that $C h_t \leq g_0 \leq C^{-1} h_t$ (see (\ref{eq:estimate for h}) and corresponding estimates in the above lemmas) on $b \leq r$. Then, estimates with respect to $| \cdot |_{h_t}$ are equivalent to those with respect to $|\cdot|_{g_0}$.
    Therefore, 
    \begin{equation}
        |\tau|_{h_t} = \left\{
        \begin{array}{cccc}
         O(\varepsilon^{2-\beta(\delta + 2)})    & 1\leq r  \\
         r^{2\delta-2} O(\varepsilon^{4 - 2\beta(\delta + 2)}) & b \leq r \leq 1 \\
         r^\delta |R(g_{\zeta(t)})|_{g_{\zeta(t)}} O (\varepsilon^{2-\beta(\delta + 2)}) 
         & \varepsilon \leq r \leq b\\
          |R(g_{\zeta(t)})|_{g_{\zeta(t)}}O(\varepsilon^{(1-\beta)(\delta +2)} )& r \le \varepsilon
        \end{array}
        \right..
    \end{equation}
    As $|R(g_{\zeta/\varepsilon^2})|_{g_{\zeta(t)/\varepsilon^2}}$ is bounded on $r \leq 1$ and has asymptotic $O(r^{-6})$ as $r \to \infty$, $|R(g_{\zeta(t)})|_{g_{\zeta(t)}}$ can be estimated as follows:
    \begin{equation}
        |R(g_{\zeta(t)})|_{g_{\zeta(t)}} = \left\{ 
        \begin{array}{cc}
         O(\varepsilon^{4}r^{-6})    &  \varepsilon \leq r \\ 
         O(\varepsilon^{-2})  & r \leq \varepsilon
        \end{array}
        \right. .
    \end{equation}
     Therefore, we have 
     \begin{equation}
        |\tau|_{h_t} = \left\{
        \begin{array}{cccc}
         O(\varepsilon^{2-\beta(\delta + 2)})    & 1\leq r  \\
         r^{2\delta-2} O(\varepsilon^{4 - 2\beta(\delta + 2)}) & b \leq r \leq 1 \\
         r^{\delta -6} O (\varepsilon^{6-\beta(\delta + 2)}) 
         & \varepsilon \leq r \leq b\\
          O(\varepsilon^{(1-\beta)(\delta +2) -2} )& r \le \varepsilon
        \end{array}
        \right..
    \end{equation}
    Finally, by integrating them, we obtain
    \begin{equation}
     \int_{X_t} \omega \wedge \tau = O(\varepsilon^{2-\beta(\delta+2)}),
     \end{equation}
     which is what we needed.
     \end{proof}
     \begin{proof}[proof of the Theorem \ref{thn:main} ]
     By summing up above lemmas, we see
     \begin{equation}
         F(t) =  \int_{X_0} f_0 c_2(g_0) + f(0)e_{orb}(X_{\dot{\zeta}}) + O(t).
     \end{equation}
     In particular, $F(0)=\int_{X_0} f_0 c_2(g_0) + f(0)e_{orb}(X_{\dot{\zeta}})$. Then, it follows that $F(t) - F(0) = O(t) $ from the above estimate. Then, $F$ is H\"older continuous with exponent $1$ on $\Delta_d \cap [0,1)$, therefore, H\"{o}lder continuous with exponent at least $\frac{1}{d}$ on $\Delta\cap [0,1)$.  
\end{proof}

\section{Polarized K3 surface case}\label{sec:K3}
In this section, we discuss a degeneration of \Kahler-Einstein metrics on a non-degenerate family of  polarized K3 surfaces. In particular, we extend the main theorem to a function defined on the whole disc.

For a polarized family of Calabi-Yau manifolds (note that K3 surface is the two dimensional Calabi-Yau manifold), the following convergence theorem is known \cite{RZ1} (see also \cite{RZ2}).
\begin{thm}[Theorem 1.4. \cite{RZ1} and its proof]\label{thm:RZ}

Let $X_0$ be an $n$ dimensional Calabi–Yau variety. Assume that $X_0$ admits a smoothing $\mathcal{X} \to \Delta$ such that $\mathcal{X}$ admits an ample line bundle $\mathcal{L}$ and the relative canonical bundle is trivial, i.e.,$K_{\mathcal{X}/\Delta }\cong \mathcal{O}_\mathcal{X}$. 
Let $\hat{g}_t$ denote the unique Ricci-flat K\"ahler metric with K\"ahler form $\hat{\omega}_t \in c_1(\mathcal{L})|_{X_t}$ and let $\{\phi_t\}$ be the solution of the equation of Calabi--Yau metric
\begin{equation}
    \left\{ \begin{array}{cc}
         &(\omega_t + dd^c \phi_t)^n =  (-1)^{\frac{n^2}{2}} S_t \Omega_t \wedge \overline{\Omega}_t\\
         & \sup_{X_t} \phi_t = 0,
    \end{array}\right.
\end{equation}
where $\Omega_t$ is a holomorphic $n$-form on $X_t$, $S_t$ is a constant depending only on $t$ and $\omega_t = \frac{1}{m}\omega_{\mathrm{FS}} |_{X_t} $ is the restriction of the Fubini--Study metric under an embedding $\mathcal{L}^m : \mathcal{X} \to \mathbb{P}^N \times \Delta$. Then, $\phi_t \to \phi_0$ on any compact subset $K \subset X_0^{reg}$ in the $C^\infty$ sense.
\end{thm}

Let $\mathcal{X} \to \Delta$ be a flat family of K3 surfaces. Assume that the fibre $X_t$ on $t \in \Delta$ is smooth for $t \neq 0$ and $X_0$ has ADE singularities. Additionally, we assume that the central fibre $X_0$ has exactly one singularity $x_0$ and the family admits the simultaneous minimal resolution $\tilde{\mathcal{X}} \to \Delta$ for simplicity. Take an ample line bundle $\mathcal{L}$ on $\mathcal{X}$ and let $\hat{g}_t$ be the Ricci-flat \Kahler\ metric with \Kahler\ form $\hat{\omega}_t \in c_1(\mathcal{L})|_{X_t}$. Fix nowhere vanishing holomorphic 2-forms $\Omega_t$ on $X_t$ so that $\hat{\omega}_t^2 = \Omega_t \wedge \overline{\Omega}_t$. We may take an open neighbourhood $\mathcal{U}$ of $x_0 \in \mathcal{X}$ as in section \ref{sec:BR}. However, we may take $\mathcal{U}$ (sufficiently small, if necessary) with better properties.

 \begin{lem}\label{lem:potential}
     Let $(\mathcal{X}, \mathcal{L}) \to \Delta$, $x_0 \in X_0$ and $\hat{g}_t$ be as above. Then there exists an open neighbourhood $\mathcal{U} \subset \mathcal{X}$ of $x_0$ such that
     \begin{itemize}
         \item $\mathcal{U} \to \Delta$ is a flat deformation of the singularity $x_0 \in X_0$ described by $\zeta(t) = (0, \sqrt{2}[\Omega_t]) : \Delta \to \mathfrak{h}_\mathbb{R} \oplus \mathfrak{h}_\mathbb{C}$,
         \item $U_0 \subset X_0$ admits a $\hat{g}_0$-adopted uniformization $\mathbb{C}^2 \supset B(0;1) \to U_0$,
         \item  the simultaneous resolution $\tilde{\mathcal{U}} \to \Delta$ induced by $\tilde{\mathcal{X}} \to \mathcal{X}$ is diffeomorphic to $U \times \Delta$ for $U = \{r < 1 \} \subset X$,
         \item $\hat{g}_t$ has a potential $r^2 + \hat{\phi}_t$ on $U_t \subset X_t$ and $\{ \hat{\phi}_t\}$ converges smoothly to $\hat{\phi}_0$ on any compact subset of  $\frac{1}{2} < r < 1$.
     \end{itemize}
 \end{lem}
\begin{proof}
    We show the fourth condition is possible, and then, take its intersection with an open neighbourhood satisfying the first three conditions. Consider an embedding 
    \begin{equation}
    \mathcal{L}^m : \mathcal{X} \to \mathbb{P}^N \times \Delta    
    \end{equation}
    with suitable $m$ and $N$. Take an affine open subset 
    \begin{equation}
    \mathbb{C}^N \times \Delta \subset \mathbb{P}^N \times \Delta
    \end{equation}
    such that $x_0 \in \mathcal{X} \cap (\mathbb{C}^N \times \Delta)$. Let $\mathcal{X}^o \subset \mathcal{X}$ be the intersection $\mathcal{X} \cap (\mathbb{C}^N \times \Delta)$. Then, $\mathcal{X}^o$ is a family of affine varieties with central fiber having a rational double point $x_0$ and $\mathcal{L}^m$ is trivial on $\mathcal{X}^o$. Then, by theorem \ref{thm:RZ}, we may take a function $\phi_t$ on $X_t$ such that $\phi_t \to \phi_0$ smoothly on $X_t \cap \{ \frac{1}{2} \leq r \leq 1\}$ and $\hat{\omega}_t = \frac{1}{m}\omega_{\mathrm{FS}}|_{X_t} + dd^c \phi_t$. Therefore, we have $\hat{\omega}_t = dd^c (\frac{1}{m}\log(1 + \sum |z_i|^2)|_{X^o_t} + \phi_t)$ on $X_t^o := X_t \cap \mathcal{X}^o \subset \mathbb{C}^N$. Then we can attain the fourth condition by taking $\hat{\phi}_t = \frac{1}{m}\log(1 + \sum |z_i|^2)|_{X^o_t} + \phi_t - r^2$.
\end{proof}

We construct an \textit{almost Ricci-flat \Kahler\ metric} on $X_t$ by gluing $\hat{g}_t$ and $g_{\zeta(t)}$. As in section \ref{sec:BR} and section \ref{sec:Main}, we assume that the family is non-degenerate in the sense of Biquard-Rollin \cite{BR} in the following extended sense:
\begin{dfn}
    Let $(\mathcal{X}, \mathcal{L}) \to \Delta$ be a polarized flat family of K3 surfaces with the central fiber with an ADE singularity $x_0 \in X_0 \cong 0 \in \mathbb{C}^2/\Gamma$. Take a suitable open neighborhood $x_0 \in \mathcal{U} \subset \mathcal{X}$ of $x_0$ which is isomorphic to an open subset of the family $\mathcal{Y}_{\zeta (t)}$ of Kronheimer's ALE gravitational instantons defined by a holomorphic function $\zeta : \Delta \to \mathfrak{h}_\mathbb{C}$. Then the family is non-degenerate at $x_0$ if $\zeta(t)$ satisfies the following condition:\\
    Let $\zeta(t) = \dot{\zeta} t^p + O(t^{p+1})$ be the expansion of $\zeta(t)$ at the origin so that $\dot{\zeta} \neq 0$. Then $\dot{\zeta}$ is not in the discriminant locus, i.e., $\dot{\zeta} \notin \theta^\perp = \{v \in \mathfrak{h}_\mathbb{C} \mid \langle v,s\rangle =0$ for any $s\in \theta \}$ for any root $\theta$.
\end{dfn}
We also need the next lemma.
\begin{lem}[Theorem 8.4.4. \cite{Joy}]
    Let $Y$ be an ALE gravitational instanton and $\tau$ be a $d$-exact real $(1,1)$ form on $Y$ with asymptotic $\| \nabla^k\tau \| = O(r^{-4-k})$. Then $\tau$ is $dd^c$-exact and a potential $\phi$ can be chosen with asymptotic $\| \nabla^k \phi \| = O(r^{-2-k})$.
\end{lem}
By applying the above lemma to $\tau = dd^c_t(r^2) - \omega_{\zeta(t)}$, we can take a potential $r^2 + \phi_{\zeta(t)}$ of $\omega_{\zeta(t)}$ on $Y$ such that $\phi_{\zeta(t)} = O(r^{-2})$. We define a \Kahler\ metric $\tilde{g}_t$ on $X_t$ by 
\begin{equation}
    \tilde{g}_t = \left\{ \begin{array}{cc}
      \hat{g}_t   & \text{on}\  X_t \backslash U  \\
       dd^c_t (r^2 + \tilde{\phi}_t)  & \text{on}\  U
    \end{array}\right.,
\end{equation}
where $\tilde{\phi}_t$ is a function on $U$ defined by
\begin{equation}
\tilde{\phi}_t = (1- \chi_{(\varepsilon^{1/2},2\varepsilon^{1/2})}(r)) \phi_{\zeta(t)} +
    \chi_{(\varepsilon^{1/2}, 2\varepsilon^{1/2})}(r)(1-\chi_{(1/2, 3/4)}(r)) \hat{\phi}_0 +
    \chi_{(1/2, 3/4)}(r) \hat{\phi}_t,
\end{equation}
where $\varepsilon = |t|^{p/2}$ and $\chi$ is the bump function defined in section \ref{sec:BR}, see (\ref{def:bump}).

\begin{center}
\begin{tikzpicture}
\draw (-1.5, -3.0) rectangle (1.5, -2);
\draw (-2.5, -3.2) rectangle (2.5, 0);
\draw (-3, -3.4) rectangle (3, 1.3);
\draw (-4.2, -3.6) rectangle (4,2.5);
\draw (-4.5, -4) rectangle (7,4.3);
\draw (-4.3 , -3.8) rectangle (4.3 ,4);

\draw (5.5,3) node{$X_t\backslash U$};
\draw (5.5,2.5) node{$\tilde{g}_t := \hat{g}_t$};

\draw (0,-2.3) node{$ U_{r\leq \varepsilon^{1/2}}$};
\draw (0,-2.8) node{$\tilde{g}_t: = g_{\zeta(t)}$};

\draw (0,-0.3) node{$U_{\varepsilon^{1/2} \leq r \leq 2\varepsilon^{1/2}}$};
\draw (0,-0.9) node{$\tilde{g}_t := dd^c_t (r^2 + \phi_{\zeta(t)})+$};
\draw (0, -1.5) node {$dd^c_t (\chi_{(\varepsilon^{1/2},2\varepsilon^{1/2})}(r)( \hat{\phi}_0 -\phi_{\zeta(t)}))$};

\draw (0,1) node{$U_{2\varepsilon^{1/2} \leq r \leq \frac{1}{2}}$};
\draw (0,0.5) node{$\tilde{g}_t := dd^c_t (r^2 + \hat{\phi}_0)$};

\draw (0,2.1) node{$ U_{\frac{1}{2} \leq r \leq \frac{3}{4}}$};
\draw (0, 1.6) node{$\tilde{g}_t := dd^c_t(r^2 + \hat{\phi_0}) + 
 dd^c_t (\chi_{(1/2, 3/4)}(r) (\hat{\phi}_t - \hat{\phi}_0))$};

 \draw(0, 3.5) node{$U_{\frac{3}{4} \leq r \leq 1}$};
 \draw(0, 3) node{$\tilde{g}_t :=dd^c_t (r^2 + \hat{\phi_t}) =  \hat{g}_t$};

\draw (0, -4.2) node{Gluing of metrics.};
\draw (7.5, 0) node{$= X_t$};
\draw (4.65, 0) node{$=U$};

\end{tikzpicture}
\end{center}

The tensor $\tilde{g}_t$ satisfies the following estimates (in particular, $\tilde{g}_t$ is actually a metric).
\begin{lem}\label{lem:comp.U}
    On the domain $X_t \backslash U_{r \leq \frac{3}{4}}$, $\tilde{g}_t$ satisfies the following estimates:
    \begin{align}
        \tilde{\omega}_t^2 &= \Omega_t \wedge \overline{\Omega}_t, \label{est:vol on X-U} \\
        \|\mathrm{Ric}(\tilde{g}_t) -\mathrm{Ric}(\hat{g}_0)\|_{\hat{g}_0} &= O(|t|), \label{est:Ricci on X-U} \\
        \| c_2 (\tilde{g}_t) - c_2(\hat{g}_0) \|_{\hat{g}_0} &= O(|t|) \label{est:c2 on X-U}.
    \end{align}
\end{lem}
\begin{proof}
    Recall that $\tilde{g}_t = \hat{g}_t$ and $\hat{g}_t \to \hat{g}_0$ smoothly on this domain and that $\hat{g}_0$ is a Ricci flat metric. The assertion follows from these facts immediately. 
\end{proof}

\begin{lem}\label{lem:1/2,3/4}
    On the domain $U_{\frac{1}{2} \leq r \leq \frac{3}{4}}$, $\tilde{g}_t$ satisfies the following estimates:
\begin{align}
   \| \tilde{g}_t - \hat{g}_0 \|_{\hat{g}_0} &= O(|t|), \label{est:g_t on 1/2, 3/4}\\
    \left|\log\left(\frac{\Omega_t \wedge \overline{\Omega}_t}{\tilde{\omega}_t^2}\right)\right| &= O(\varepsilon^4) \label{est:f on 1/2, 3/4}, \\
   \|\mathrm{Ric}(\tilde{g}_t) -\mathrm{Ric}(\hat{g}_0)  \|_{\hat{g}_0} &= O(|t|),\label{est:ricci for 1/2, 3/4} \\
   \|c_2(\tilde{g}_t) - c_2(\hat{g}_0) \|_{\hat{g}_0} &= O(|t|). \label{est:c_2 for 1/2, 3/4}
\end{align}
\end{lem}

\begin{proof}
    By the construction of $\hat{\phi}_t$ and $\hat{g}_t \to \hat{g}_0$ smoothly, we have
    \begin{equation}
    \begin{aligned}
        dd^c_t(r^2 + \hat{\phi}_0) &= dd^c_t(r^2 + \hat{\phi}_t) + dd^c_t(\hat{\phi}_0 - \hat{\phi}_t) \\
        &= \hat{g}_t + O(|t|) = \hat{g}_0 + O(|t|),
    \end{aligned}
    \end{equation}
    \begin{equation}
        dd^c_t (\chi_{(\frac{1}{2}, \frac{3}{4})}(r) (\hat{\phi}_t - \hat{\phi}_t)) = O(|t|).
    \end{equation}
    In particular, we obtain the first assertion (\ref{est:g_t on 1/2, 3/4}). For (\ref{est:f on 1/2, 3/4}), recall that for the complex structure $I_t = I_{\zeta(t)}$, we have the following estimate (see lemma (\ref{lem:cpx.str})) :
    \begin{equation}
    \begin{aligned}
        I_{\zeta(t)} &= I_0 + I_\zeta^{(2)} r^{-4} + O(r^{-6})\\
        &= I_0 + O(\varepsilon^4),
    \end{aligned}
    \end{equation}
    where $I^{(2)}_\zeta$ is a homogeneous polynomial in $\zeta$ of degree $2$ (recall that $\zeta(t) = t^p \dot{\zeta} + O(t^{p+1})$ and $\varepsilon^2 = |t|^p$). By direct calculation,  we have the following:
    \begin{equation}
        \begin{aligned}
            \tilde{\omega}_t^2 &= (\hat{g}_t + dd^c_t( (\hat{\phi}_0 - \hat{\phi}_t) + \chi_{(1/2,3/4)}(r) (\hat{\phi}_t - \hat{\phi}_0))^2\\
            &= \hat{\omega}_t^2 + \hat{\omega}_t \wedge dd^c_t( (\hat{\phi}_0 - \hat{\phi}_t) + \chi_{(1/2,3/4)}(r) (\hat{\phi}_t - \hat{\phi}_0)) \\
            &+ (dd^c_t( (\hat{\phi}_0 - \hat{\phi}_t) + \chi_{(1/2,3/4)}(r) (\hat{\phi}_t - \hat{\phi}_0)))^2 .
        \end{aligned}
    \end{equation}
    The first term is equal to $\Omega_t \wedge \overline{\Omega}_t$. For the second and the third terms, we use the following estimates:
    \begin{align}
        (\hat{\omega}_t - \hat{\omega}_0)(\hat{\omega}_t + \hat{\omega}_0) &= \Omega_t \wedge \overline{\Omega}_t - \Omega_0 \wedge \overline{\Omega}_0 = O(\varepsilon^4) \Omega_0 \wedge \overline{\Omega}_t,
    \end{align}
    and there exists a constant $C >0$ independent of $t$ such that 
    \begin{align}
        C^{-1} \| dd^c_t (\hat{\phi}_t - \hat{\phi}_0) \|_{\hat{g}_t} <
        \|dd^c_t( (\chi_{(1/2,3/4)}(r) &-1)(\hat{\phi}_t - \hat{\phi}_0))\|_{\hat{g}_t} < 
        C \| dd^c_t (\hat{\phi}_t - \hat{\phi}_0) \|_{\hat{g}_t}, \\
        C^{-1} \hat{g}_0 < &\hat{g}_t < C \hat{g}_0.
    \end{align}
    Then, we obtain the followings:
    \begin{align}
        \hat{\omega}_t \wedge dd^c_t(\hat{\phi}_0 - \hat{\phi}_t) &= O(\varepsilon^4) \hat{\omega}_t^2\\
        (dd^c_t (\hat{\phi}_t - \hat{\phi}_0))^2 &= O(\varepsilon^4) \hat{\omega}_t^2.
    \end{align}
Therefore, $\log\left(\frac{\Omega_t \wedge \overline{\Omega}_t}{\tilde{\omega}_t^2}\right) = O(\varepsilon^4)$ on this domain. The last assertion (\ref{est:c_2 for 1/2, 3/4}) can be confirmed by the similar arguments as in section \ref{sec:Main} (namely, the proofs of lemma \ref{lem:4b1} for example). Explicitly, direct calculation yields 
\begin{align}
    \tilde{g}_t ^{-1}&= \hat{g}_0^{-1} + O(|t|) \\
    \partial_t \tilde{g}_t &= \partial_0 \hat{g}_0 + O(|t|) \\
    \partial^2_t \tilde{g}_t &= \partial^2_0 \hat{g}_0 + O(|t|),
\end{align}
 under the trivialization by $\{e^1_t, e^2_t\}$, where $\partial_t$ denotes any differential by the vector fields $\{e^1_t, e^2_t\}$ defined in Lemma \ref{lem:cpx.str}. Then we have (\ref{est:ricci for 1/2, 3/4}) (\ref{est:c_2 for 1/2, 3/4}).
\end{proof}

\begin{lem}\label{lem:2 1/2, 1/2}
    On the domain $U_{2\varepsilon^{1/2} \leq r \leq \frac{1}{2}}$, $\tilde{g}_t$ satisfies the following estimates:
    \begin{align}
         \| \tilde{g}_t - \hat{g}_0 \|_{\hat{g}_0} &= O(\varepsilon^4r^{-4}), \label{est:g_t on 2, 1/2}\\
   \log\left(\frac{\Omega_t \wedge \overline{\Omega}_t}{\tilde{\omega}_t^2}\right) &= O(\varepsilon^4r^{-4}), \label{est:f on 2, 1/2} \\
   \|\mathrm{Ric}(\tilde{g}_t) -\mathrm{Ric}(\hat{g}_0)  \|_{\hat{g}_0} &= O(\varepsilon^4r^{-4}), \label{est:ricci for 2, 1/2} \\
   \|c_2(\tilde{g}_t) - c_2(\hat{g}_0) \|_{\hat{g}_0} &= O(\varepsilon^4r^{-4}). \label{est:c_2 for 2, 1/2}
    \end{align}
\end{lem}
\begin{proof}
    On this domain, we have the following estimate on the complex structure $I_{\zeta(t)}$:
    \begin{equation}
        I_{\zeta(t)} = I_0 + O(\varepsilon^4 r^{-4}).
    \end{equation}
    In particular, we have 
    \begin{equation}
    \begin{aligned}
        \tilde{g}_t &= dd^c_t(r^2 + \hat{\phi}_0) \\
       &= dd^c_0 (r^2 + \hat{\phi}_0) + O(\varepsilon^4 r^{-4})\\
        &= \hat{g}_0 + O(\varepsilon^4r^{-4}).
    \end{aligned}
    \end{equation}
    Therefore, 
    \begin{equation}
        \begin{aligned}
            \tilde{\omega}_t^2 &= \hat{\omega}_0^2 + O(\varepsilon^4r^{-4}) \Omega_t \wedge \overline{\Omega}_t \\
            &= (1 + O(\varepsilon^4r^{-4})) \Omega_t \wedge \overline{\Omega}_t.
        \end{aligned}
    \end{equation}
    Again, (\ref{est:ricci for 2, 1/2}) and (\ref{est:c_2 for 2, 1/2}) follows from a direct calculation as in the proof of Lemma \ref{lem:1/2,3/4}.
\end{proof}
\begin{lem}\label{lem:1,2}
On $U_{\varepsilon^{1/2} \leq r \leq 2\varepsilon^{1/2}}$, $\tilde{g}_t$ satisfies the following estimates:
\begin{align}
    \tilde{g}_t &= g_{\zeta(t)} + O(\varepsilon), \label{est: g_t on 1,2}\\
    \log \left( \frac{\Omega_t \wedge \overline{\Omega}_t}{\tilde{\omega}_t^2} \right) &= O(\varepsilon^2) , \label{est:f on 1,2} \\
    \|\mathrm{Ric}(\tilde{g}_t) -\mathrm{Ric}(\hat{g}_0)  \|_{\hat{g}_0} &= O(\varepsilon), \label{est:ricci for 1,2}\\
    c_2(\tilde{g}_t) &= c_2 (g_{\zeta(t)}) + O(\varepsilon).\label{est: c_2 for 1,2}
\end{align}
\end{lem}
\begin{proof}
    By direct calculation, we have the following:
\begin{equation}
    \begin{aligned}
        dd^c_t(\chi_{(\varepsilon^{1/2},2\varepsilon^{1/2})} (r) (\hat{\phi}_0 - \phi_{\zeta(t)}))
        &= \frac{\chi'_{(\varepsilon^{1/2},2\varepsilon^{1/2})}(r) (\hat{\phi}_0 - \phi_{\zeta(t)})}{\varepsilon^{1/2}} dd^c_t r \\
        &+ \frac{\chi''_{(\varepsilon^{1/2},2\varepsilon^{1/2})}(r) (\hat{\phi}_0 - \phi_{\zeta(t)})}{\varepsilon} dr \wedge d^c_t r\\
        &+ \frac{\chi'_{(\varepsilon^{1/2},2\varepsilon^{1/2})}(r)}{\varepsilon^{1/2}}( d(\hat{\phi}_0 - \phi_{\zeta(t)}) \wedge d^c_t r + dr \wedge d_t^c (\hat{\phi}_0 - \phi_{\zeta(t)}))\\
        &+ \chi_{(\varepsilon^{1/2},2\varepsilon^{1/2})}(r) dd^c_t(\hat{\phi}_0 - \phi_{\zeta(t)}).
    \end{aligned}
\end{equation}
As $\hat{\phi}_0 = O(r^4), \phi_{\zeta(r)} = O(\varepsilon^{4} r^{-4})$ and $ r = O(\varepsilon^{1/2})$, we have 
\begin{equation}
    dd^c_t(\chi_{(\varepsilon^{1/2},2\varepsilon^{1/2})} (r) (\hat{\phi}_0 - \phi_{\zeta(t)})) = O(\varepsilon)
\end{equation}
as $|t| \to 0$. In particular, we have 
\begin{equation}
    \tilde{g}_t = g_{\zeta(t)} + O(\varepsilon).
\end{equation}
By direct calculation,  we have the following:
    \begin{equation}
        \begin{aligned}
            \tilde{\omega}_t^2 &= (\omega_{\zeta(t)} + dd^c_t(\chi_{(\varepsilon^{1/2},2\varepsilon^{1/2})}(r) (\hat{\phi}_0 -\phi_{\zeta(t)})))^2\\
            &= \Omega_t \wedge \overline{\Omega}_t + 2\omega_{\zeta(t)} \wedge dd^c_t( \chi_{(\varepsilon^{1/2}, 2\varepsilon^{1/2})}(r) (\hat{\phi}_0 - \phi_{\zeta(t)})) \\
            &+ (dd^c_t(\chi_{(\varepsilon^{1/2}, 2\varepsilon^{1/2})}(r) (\hat{\phi}_0 - \phi_{\zeta(t)})))^2 .
        \end{aligned}
    \end{equation}
    For the second and the third terms, we use the following estimates:
    \begin{align}
        (\omega_{\zeta(t)} - \hat{\omega}_0)(\omega_{\zeta(t)} + \hat{\omega}_0) &= \Omega_t \wedge \overline{\Omega}_t - \Omega_0 \wedge \overline{\Omega}_0 = O(\varepsilon^2) \Omega_0 \wedge \overline{\Omega}_t,
    \end{align}
    and there exists a constant $C >0$ independent of $t$ such that 
    \begin{align}
        C^{-1}\| dd^c_t (\hat{\phi}_0 - \phi_{\zeta(t)}) \|_{\hat{g}_0} <&
        \|dd^c_t( (\chi_{(\varepsilon^{1/2}, 2\varepsilon^{1/2})}(r))(\hat{\phi}_0 - g_{\zeta(t)}))\|_{\hat{g}_0} < 
        C \| dd^c_t (\hat{\phi}_0 - g_{\zeta(t)}) \|_{\hat{g}_0}, \\
        C^{-1} \hat{g}_0 < &g_{\zeta(t)} < C \hat{g}_0.
    \end{align}
    Then, we obtain the followings:
    \begin{align}
        \hat{\omega}_t \wedge dd^c_t(\hat{\phi}_0 - \hat{\phi}_t) &= O(\varepsilon^2) \hat{\omega}_t^2\\
        (dd^c_t (\hat{\phi}_t - \hat{\phi}_0))^2 &= O(\varepsilon^2) \hat{\omega}_t^2.
    \end{align}
Therefore, $ \log\left(\frac{\Omega_t \wedge \overline{\Omega}_t}{\tilde{\omega}_t^2}\right) = O(\varepsilon^2)$ on this domain.
\end{proof}

For later use, we also give estimates for \textit{rescaled data}.
\begin{lem}\label{lem:estimates for rescaled data}
    Let $\tilde{g}_t, \hat{g}_t, f_t$ are as above. Consider the following rescaled data:
    \begin{equation}
    \begin{array}{cc}
      \tilde{h}_t := \frac{1}{\varepsilon^2}H^*_{(\zeta/\varepsilon^2, \varepsilon)}\tilde{g}_t   &  and\\
      k_t := H^*_{(\zeta/\varepsilon^2, \varepsilon)}f_t ,
    \end{array}
    \end{equation}
    on the domain $Y_t := Y_{\zeta/\varepsilon^2, r < 1/\varepsilon}$.
  Then, the following estimates hold:
  \begin{equation}
      \begin{aligned}
         \sup_{Y_t} \|\mathrm{Rm}(\tilde{h}_t) \|_{\tilde{h}_t} < C, \\
          \|k_t\|_{C^1(Y_t)} <  C \varepsilon^2,
      \end{aligned}
  \end{equation}
  for some constant $C>0$ independent of t.
\end{lem}

\begin{proof}
Note that
\begin{equation}\label{eq:representation of k in terms of the hol 2 form and the kahler class}
    k_t = \log \left( \frac{\frac{1}{\varepsilon^2}H^*_{(\zeta/\varepsilon^2, \varepsilon)}\Omega_t \wedge \frac{1}{\varepsilon^2}H^*_{(\zeta/\varepsilon^2, \varepsilon)}\overline{\Omega}_t}{\tilde{\varpi}^2_t} \right),
\end{equation}
    where $\tilde{\varpi}_t$ is the K\"ahler form of $\tilde{h}_t$.
We take a similar strategy to the above lemmas.
Namely, for the domain $ \{ r < \varepsilon^{-\frac{1}{2}}\}$ (corresponds to $\{ r< \varepsilon^{\frac{1}{2}}\}$ in $X_t$), we have $\tilde{h}_t = g_{\zeta(t)/\varepsilon^2}$ which converges to $g_{\dot{\zeta}}$ locally uniformly in $C^\infty$ sense.
Additionally, they are ALE gravitational instantons then by taking a compact subset $K$ of $Y$ (the underlying differentiable manifold of $Y_\zeta$ s), the curvatures are of $O(r^{-6})$ outside of $K$ uniformly with respect to $t$.
Then the norm of curvature is uniformly bounded in this domain.
By construction of $\Omega_t$, we have $\tilde{\varpi}_t^2 = \omega_{\zeta/\varepsilon^2}^2 = \varepsilon^{-4} H^*_{(\zeta/\varepsilon^2, \varepsilon)}(\Omega_t \wedge \overline{\Omega}_t)$. 
Therefore, $k_t \equiv 0$ on this domain.

For $\{\varepsilon^{-\frac{1}{2}} < r < 2\varepsilon^{-\frac{1}{2}}\} $, we see that
\begin{equation}\label{est: estimate of tilde h sub t on varepsilon to - 1/2 < r < 2 varepsilon to 1/2 }
    \tilde{h}_t = g_{\zeta/\varepsilon^2} + O(\varepsilon^3)
\end{equation}
by the similar calculation to the lemma \ref{lem:1,2}. 
And for the holomorphic $2$-form, we have
\begin{equation}
    \frac{1}{\varepsilon^2} H^*_{(\zeta/\varepsilon^2, \varepsilon)} \Omega_t =\frac{1}{\varepsilon^2} H^*_{(\dot{\zeta}, \varepsilon)} \Omega_0 + O(\varepsilon^2).
\end{equation}
Then, in particular, we have $\|\nabla k_t\| = O(\varepsilon^2)$ on this domain.

For $\{ 2\varepsilon^{-\frac{1}{2}} < r < \frac{1}{2}\varepsilon^{-1}\}$, we have 
\begin{equation}\label{est: estimate of tilde h sub t on varepasilon to -1/2 < r < varepsilon to -1}
    \tilde{h}_t = \frac{1}{\varepsilon^2}H^*_{(\zeta/\varepsilon^2, \varepsilon)}\hat{g}_t + O(\varepsilon^2)
\end{equation}
by the similar calculation as the lemma \ref{lem:2 1/2, 1/2}. 
Also, for the holomorphic $2$-form, we have 
\begin{equation}
    \frac{1}{\varepsilon^2} H^*_{(\zeta/\varepsilon^2, \varepsilon)} \Omega_t =\frac{1}{\varepsilon^2} H^*_{(\dot{\zeta}, \varepsilon)} \Omega_0 + O(\varepsilon^2)
\end{equation}
by a similar calculation as the lemma \ref{lem:2 1/2, 1/2}. This implies $\|\nabla k_t \| = O(\varepsilon^2)$ on this domain.
\end{proof}
We also need the next lemma to derive a fine $C^0$ estimates.
\begin{lem}\label{lem:Lp estimate for f}
    \begin{equation}\label{ineq:Lp estimate for f}
        \|f_t\|_{L^p_{\tilde{g}_t}} < C\varepsilon^{2+ \frac2p}
    \end{equation}
    for any $1 < p \leq \frac{4}{3}$ with some constant $C>0$ independent of $p$ and $t$
\end{lem}
\begin{proof}
    As $f \equiv 0$ on $X_t \backslash U_{r<3/4}$ and $U_{r \leq \varepsilon^{1/2}}$, we see that
    \begin{equation}
        \int_{X_t} |f_t|^p \mathrm{Vol}_{\tilde{g}_t} = \int_{U_{\varepsilon^{1/2} \leq r \leq 2\varepsilon^{1/2}}} |f_t|^p + \int_{U_{2\varepsilon^{1/2}\leq r \leq 1/2}} |f_t|^p + \int_{U_{1/2 \leq r \leq 3/4}} |f_t|^p.
    \end{equation}
    For the first term and the third term, by the estimates (namely inequalities \ref{est:f on 1/2, 3/4} and \ref{est:f on 1,2}), we obtain
    \begin{equation}
        \begin{aligned}
            \int_{U_{\varepsilon^{1/2} \leq r \leq 2\varepsilon^{1/2}}} |f_t|^p < C\varepsilon^{2p + 2}, \\
            \int_{U_{1/2 \leq r \leq 3/4}} |f_t|^p < C\varepsilon^{4p}
        \end{aligned}
    \end{equation}
    for some constant $C >0$ independent of $1\leq p$ and $t$. For the second term, we have
    \begin{equation}
    \begin{aligned}
        \int_{U_{2\varepsilon^{1/2}\leq r \leq 1/2}} |f_t|^p &< C\varepsilon^{4p}\int_{2\varepsilon^{1/2}}^{1/2} r^{-4p + 3} dr \\
        &<C \varepsilon^{2p + 2}
    \end{aligned}
    \end{equation}
    for some constant $C >0$ independent of $1 <p \leq 4/3$ and $t$.
    Therefore, we have 
    \begin{equation}
        \|f_t\|_p < C \varepsilon^{2 + \frac{2}{p}}.
    \end{equation}
\end{proof}

By the above lemmas, in particular the estimates (\ref{est:g_t on 1/2, 3/4}), (\ref{est:g_t on 2, 1/2}) and (\ref{est: g_t on 1,2}), we see that $\tilde{g}_t$ is a metric on $X_t$. As $\tilde{g}_t$ is a \Kahler\ metric on $X_t$ cohomologous to $\hat{g}_t$, we may consider the following Monge--Amp\`ere equation using $\tilde{g}_t$ as the initial data:

\begin{equation}\label{eq:MA}
    \left\{ \begin{array}{cc}
        (\tilde{\omega}_t + dd^c_t u_t)^2 = e^{f_t} \tilde{\omega}_t^2&  \\
        \int_{X_t} u_t \tilde{\omega}_t^2 = 0 & 
    \end{array}\right.,
\end{equation}
where $f_t = \log \left( \frac{\Omega_t \wedge \overline{\Omega}_t}{\tilde{\omega}_t^2}\right)$. Recall that $\Omega_t$ is the nowhere vanishing holomorphic 2-form on $X_t$ such that $\Omega_t \wedge \overline{\Omega}_t = \hat{\omega}_t^2$. By the above lemmas, in particular the estimates (\ref{est:f on 1/2, 3/4}), (\ref{est:f on 2, 1/2}) and (\ref{est:f on 1,2}), we see that $f_t$ is $\varepsilon^2$ times a smooth function on $X_t$ uniformly bounded with respect to $t$.

Using the estimate of $f_t$, we derive a $C^0$-estimate on $u_t$ by following \cite{Kob}.
\begin{prop}
    Let $u_t$ be the solution of the equation (\ref{eq:MA}). Then we have 
    \begin{equation}
        \|u_t\|_{C^0(X_t)} < C \varepsilon^4
    \end{equation}
    for some constant $C>0$ independent of $t$.
\end{prop}

\begin{proof}
    We prove it by Moser's iteration argument.
    By assumption that $u_t$ is the solution, we have
    \begin{equation}
    \begin{array}{ll}
        (1-e^{f_t})\tilde{\omega}_t^2 
        &= \tilde{\omega}_t^2 - \hat{\omega}_t^2\\
        &= -(dd^c_t u_t) \wedge (\tilde{\omega}_t + \hat{\omega}_t).
    \end{array}
    \end{equation}
    Multiplying $|u_t|^{p-2}u_t$ both side (for $p>1$) and integration by parts yields that
    \begin{equation}
      \int_{X_t} |1- e^{f_t}||u_t|^{p-2}u_t \tilde{\omega}_t^2 \geq 
    \frac{4(p-1)}{p^2} \int_{X_t} \sqrt{-1} \partial |u_t|^{\frac{p}{2}} \wedge \overline{\partial} |u_t|^{\frac{p}{2}} \wedge \tilde{\omega}_t .
    \end{equation}
     Then, as $|f_t|$ is sufficiently small (namely, $|f_t| < C \varepsilon^2 $ for some $C >0$ independent of $t$), we have the following:
\begin{equation}\label{ineq:moser}
    \int_{X_t} \| d|u_t|^{\frac{p}{2}}\|^2_{\tilde{g}_t} \mathrm{Vol}_{\tilde{g}_t}
    \leq pC \int_{X_t} |f_t| |u_t|^{p-1} \mathrm{Vol}_{\tilde{g}_t}.
\end{equation}
Here, we use the Sobolev inequality on manifolds:
\begin{lem}[\cite{Gallot}]
    Let $(X,g)$ be a $m$-dimensional compact Riemannian manifold. Assume that
    \begin{equation}\label{cond:Dependence of Sobolev constant}
        \mathrm{diam}(X,g)^2 \mathrm{Ric}(g) \geq -\alpha^2g
    \end{equation}
    for some constant $\alpha >0$. Then there exists a constant $\kappa= \kappa(m,\alpha) > 0$ such that
    \begin{equation}\label{ineq:Sobolev inequality}
        \kappa\frac{\mathrm{Vol}(X,g)^{\frac{1}{m}}}{\mathrm{diam}(X,g)}\|u\|_{L^{\frac{2m}{m-2}}} \leq \|du\|_{L^2}
    \end{equation}
    for any $u \in C^\infty(X) $ with $\int_X u \mathrm{Vol}_g =0$.
\end{lem}
As the diameters, volumes and Ricci curvatures of $\{\tilde{g}_t\}$ are uniformly bounded, we have a uniform Sobolev constant. Then, applying Sobolev inequality to the left hand side of (\ref{ineq:moser}) with $p=2$ and the H\"older inequality for the right hand side, we obtain
\begin{equation}
 \| u_t \|_{L^4}^2 \leq C \| f_t \|_{L^q} \| u_t \|_{L^r}
\end{equation}
for any $q,r>1$ with $\frac{1}{q} + \frac{1}{r}=1$. By taking $r = 4$, we have
\begin{equation}
    \| u_t\|_{L^4} \leq C \varepsilon^{2 + \frac{2}{q}},
\end{equation}
where $C > 0$ is independent of $t$. Therefore, by Moser's iteration argument, we have the $C^0$ estimate 
\begin{equation}\label{est:C0}
    \|u_t\|_{C^0} \leq C \varepsilon^4.
\end{equation}
\end{proof}

  As the curvature of $\tilde{g}_t$ diverges on $U_t$ as $t \to 0$, we consider the rescaled metrics $\hat{h}_t = \frac{1}{\varepsilon^2} H^*_{(\varepsilon, \zeta/\varepsilon^2)}\hat{g}_t$ and $\tilde{h}_t = \frac{1}{\varepsilon^2} H^*_{(\varepsilon,\zeta/\varepsilon^2)}\tilde{g}_t$ on $U_t$ to get $C^2$ and higher order estimates. For these rescaled metrics, we have the following equation:
\begin{equation}
        (\tilde{\varpi}_t + dd^c_t v_t)^2 = e^{k_t} \tilde{\varpi}_t^2,
\end{equation}
  where $v_t = \frac{1}{\varepsilon^2} H^*_{(\varepsilon, \zeta/\varepsilon^2)} u_t, k_t = H^*_{(\varepsilon,\zeta/\varepsilon^2)}f_t$ and $\tilde{\varpi}_t$  the \Kahler\ form of $\tilde{h}_t$. Note that $(Y_{\zeta/\varepsilon, r<1/\varepsilon}, \tilde{h}_t)$ converges to $Y_{\dot{\zeta}}$ locally uniformly in $C^\infty$ sense (by non-degeneracy assumption), therefore, $\tilde{h}_t$ has uniformly bounded curvatures and injectivity radius that approaches to those of the Euclidean metric outside of a compact subset. 
For $C^2$ estimate of $v_t$, we use the next equality .
\begin{prop}[\cite{Sze} Lemma 3.7]
    Let $(X,g)$ be a \Kahler\ manifold (not necessarily compact!) of dimension $n$. Then, for any \Kahler\ metric $\omega' = \omega + dd^c u$, the following inequality holds:
    \begin{equation}\label{prop:lem for C2 estimate}
        \Delta_{g'} (\log (\mathrm{Tr}_{g}g')) \geq A\mathrm{Tr}_{g'}g - \frac{g^{j\overline{k}}R'_{j\overline{k}}}{\mathrm{Tr}_{g}g'},
    \end{equation}
    where $R'$ is the Ricci curvature of $g'$ and $A >0$ is a constant depending on the bisectional curvature of $g$.
\end{prop}

We apply the above inequality for $g = \tilde{h}_t$ and $u = v_t$ (hence $g' = \hat{h}_t$). As $\hat{h}_t$ converges to an ALE gravitational instanton $Y_{\dot{\zeta}}$, we may take $A$ independent of $t$. Then we have
\begin{equation}\label{ineq:schwart type inequality}
    \Delta_{\hat{h}_t} (\mathrm{log}(\mathrm{Tr}_{\tilde{h}_t} \hat{h}_t)) \geq  A \mathrm{Tr}_{\hat{h}_t} \tilde{h}_t
\end{equation}
as $\hat{h}_t$ is a Ricci flat metric.
As $u_t$ is pluriharmonic and $ \log(\mathrm{Tr}_{\tilde{g}_t} \hat{g}_t)=\mathrm{const.}$ on $3/4 <r$, $ \max_{X_t} (\log(\mathrm{Tr}_{\tilde{g}_t} \hat{g}_t) - Au_t) = \max_{U}(\log(\mathrm{Tr}_{\tilde{g}_t} \hat{g}_t) - Au_t) $. Assume $\log (\mathrm{Tr}_{\tilde{h}_t} \hat{h}_t) - Av_t$ attains a maximum value at some point $p \in Y_{\zeta(t)/\varepsilon^2, r \leq 1/\varepsilon}$. 
Then we have 
\begin{equation}
\begin{aligned}
     0 &\geq \Delta_{\hat{h}_t} (\log(\mathrm{Tr}_{\tilde{h}_t} \hat{h}_t) - Au_t )(p)\\
     &\geq 2A(\mathrm{Tr}_{\hat{h}_t} \tilde{h}_t (p) -1)
\end{aligned} 
\end{equation}
In particular, we see that 
\begin{equation}
    \frac{1}{2A} \geq (\mathrm{Tr}_{\tilde{h}_t} \hat{h}_t) (p)
\end{equation}
with a constant $A>0$ independent of $t$. Therefore, we obtain
\begin{equation}
    \sup_{Y_{\zeta(t)/\varepsilon^2, r < 1/\varepsilon}} \mathrm{Tr}_{\tilde{h}_t} \hat{h}_t \leq \exp{(A_1 + A_2 \|v_t\|_\infty)},
\end{equation}
with constants $A_i$ independent of $t$. Therefore, we have a $C^2$ estimate of $v_t$:
\begin{equation}
    \| dd^c_t v_t \|_{\tilde{h}_t} \leq C,
\end{equation}
where $C > 0$ is a constant independent of $t$ on $Y_{\zeta/\varepsilon^2, r < 1/\varepsilon}$. 
As the same, applying (\ref{ineq:schwart type inequality}) for $g= \tilde{g}_t$ and $g'=\hat{g}_t$ , we have a uniform $C^2$ estimate on $X_t \backslash U$:
\begin{equation}
    \sup_{X_t\backslash U} \mathrm{Tr}_{\tilde{g}_t} \hat{g}_t \leq \mathrm{exp}(A_1 + A_2\|u_t\|_\infty) <C.
\end{equation}

\begin{prop}
    In the above situation, we have the following $C^2$ estimates:
    \begin{equation}\label{ineq:C^2 estimates of the MA equations}
        \begin{array}{cc}
           \| dd^c_t v_t \|_{\tilde{h}_t} \leq C  & and \\
            \| dd^c_t u_t \|_{\tilde{g}_t} \leq C, & 
        \end{array}
    \end{equation}
    for some constant $C>0$ independent of $t$.
\end{prop}

By the general theory of non-linear elliptic equation, $v_t$ automatically has uniformly bounded $C^{2,\alpha}_{\tilde{h}_t}$ norm for some $0< \alpha <1$ depending on $\tilde{h}_t, u_t, \partial u_t, \partial^2 u_t$ (see \cite{GT} Theorem 17.14 and \cite{Siu} Chapter 2 section 4).
\begin{prop}\label{Prop:C2 alpha estimate for the solution}
    In the above situation, the followings holds:
    \begin{align}
        \|u_t \|_{C_{\tilde{g}_t}^{2,\alpha}(X_t \backslash U)} < C, \\
        \|v_t\|_{C_{\tilde{h}_t}^{2,\alpha}(Y_{\zeta/\varepsilon^2, r<1/\varepsilon})} < C,
    \end{align}
    for some constants $C > 0$ and $0 < \alpha < 1$ independent of $t$.
\end{prop}
We now consider the following \textit{deformation} of equation (\ref{eq:MA}):
\begin{equation}\label{eq:sMA}
   \left\{ \begin{array}{cc}
        (\tilde{\omega}_t + dd^c_t u_{t,s})^2 = e^{sf_t}\tilde{\omega}_t^2&  \\
        \int_{X_t} u_{t,s} \tilde{\omega}_t^2 = 0 & 
    \end{array}\right.,
\end{equation}
where $s \in [0,1]$ is a parameter. Note that we obtain the original equation by letting $s=1$ and obtain an equation which has a trivial solution by letting $s=0$.
By differentiating the equation with respect to $s$, we have 
\begin{equation}\label{eq:linearlized sMA}
\left\{\begin{array}{cc}
     & \Delta_{t,s} \frac{\partial u_{t,s}}{\partial s} = f_t  \\
     & \int_{X_t} \frac{\partial u_{t,s}}{\partial s} \mathrm{Vol}_{g_{t,s}} =0
\end{array}\right.
     ,
\end{equation}
where $\Delta_{t,s}$ is the Laplacian of the \Kahler\ metric $\tilde{g}_t + dd^c_t u_{t,s}$. Also, we consider the linearization of the rescaled equation:

\begin{equation}\label{eq:linearized rescaled sMA}
      \tilde\Delta_{t,s}\left( \frac{\partial v_{t,s}}{\partial s} \right)= k_t  
     .
\end{equation}
We obtain the following estimates by the same arguments as above but $f_t$ being replaced by $sf_t$:
\begin{align}
    \|u_{t,s}\|_{C^0(X_t)} < C \varepsilon^4, \label{ineq:C 0 estimate for u t s} \\
    \|u_{t,s}\|_{C^{2,\alpha}_{\tilde{g}_t}(X_t\backslash U)} < C,\label{ineq:C 2 alpha estimate for u t s}\\
    \|v_{t,s}\|_{C^{2,\alpha}_{\tilde{h}_t}(Y_{\zeta/\varepsilon^2, r < 1/\varepsilon})}<C\label{ineq: C 2 alpha estimate for v t s},
\end{align}
for some constants $C>0$ and $0<\alpha <1$ independent of $t$ and $s$.
We use the Schauder estimate for linear elliptic equations to derive a fine $C^2$ estimate for $\frac{\partial u_{t,s}}{\partial s}$.

\begin{prop}[\cite{GT} Theorem 6.2]
Let $\Omega \subset \mathbb{R}^n$ be a bounded domain and let $u \in C^{2,\alpha}(\Omega)$ be a bounded solution of the equation 
\begin{equation*}
    Lu = f,
\end{equation*}
where $L = \sum_{l,k} a_{j,k} \partial_j \partial_k + \sum_l b_l \partial_l + c$ is a strictly elliptic operator with $(a_{j,k}) \geq \lambda >0$ (as a quadratic form). Assume that 
\begin{itemize}
    \item $f$ is $\alpha$- H\"older continuous on $\Omega$ and
    \item $\|a_{l,k}\|_{C^{0,\alpha}(\Omega)}, \|b_l\|_{C^{0,\alpha}(\Omega)}, \|c\|_{C^{0,\alpha}(\Omega)} \leq \Lambda$.
\end{itemize}
Then, for any open subset $\Omega' \subset \subset \Omega$, there exists a (positive) constant $C= C(n, \alpha, \lambda, \Lambda, \mathrm{diam}({\Omega}), \mathrm{dist}(\Omega', \partial \Omega))$ so that 
\begin{equation}\label{eq:schauder}
    \|u\|_{C^{2,\alpha}(\Omega')} \leq C(\|u\|_{C^0(\Omega)} + \|f\|_{C^{0,\alpha}(\Omega)}).
\end{equation}
\end{prop}
First, we apply the estimate for the equation (\ref{eq:linearlized sMA}). 
\begin{prop}
    Let $\frac{\partial u_{t,s}}{\partial s}$ be a solution of (\ref{eq:linearlized sMA}) on the domain $\Omega = X_t \backslash U_{r \leq 1}$. Then, 
    \begin{equation}\label{est; u dot}
        \left\|\frac{\partial u_{t,s}}{\partial s}\right\|_{C^{2,\alpha}(\Omega)} \leq C \varepsilon^2,
    \end{equation}
    for some constant $C>0$ independent of $s$ and $t$.
\end{prop}
\begin{proof}
    Apply the inequality (\ref{eq:schauder}). 
    In this case, $a_{i,j} = (\tilde{g} + dd^c_tu_{t,s})^{i,j}$, which we already have a uniform estimate (\ref{ineq:C 2 alpha estimate for u t s}), and $b_l = c =0$. Also, note that multiplying $\frac{\partial u_{t,s}}{\partial s}\left| \frac{\partial u_{t,s}}{\partial s} \right|^{p-2}$ both side of the equation (\ref{eq:linearlized sMA}) yields the inequality which is same form as (\ref{ineq:moser}). By the estimate (\ref{ineq:C 2 alpha estimate for u t s}) and an equality $\mathrm{Ric}(\tilde{g}_t + dd^c_t u_{t,s}) = s\mathrm{Ric}(\tilde{g}_t)$, we have a uniform (with respect to $t$ and $s$) Sobolev constant for $\tilde{g}_t + dd^c_t u_{t,s}$ hence we have a uniform $C^0$ estimate $\left\| \frac{\partial u_{t,s}}{\partial s} \right\|_{C^0(X_t)} <C\varepsilon^4$ by the Moser iteration argument again. Then we have a uniform estimate 
    \begin{equation*}
        \left\| \frac{\partial u_{t,s}}{\partial s} \right\|_{C^{2,\alpha}(\Omega)} \leq  C \varepsilon^4 + \|f_t\|_{C^{0,\alpha}(\Omega)} \leq C\varepsilon^2.
    \end{equation*}
\end{proof}
\begin{prop}
    Let $\frac{\partial v_{t,s}}{\partial s}$ be a solution of (\ref{eq:linearized rescaled sMA}). Then, 
    \begin{equation}
           \left|\frac{\partial v_{t,s}}{\partial s} \right|_{C^{2,\alpha}_{\tilde{h}_t}(Y_{\zeta/\varepsilon^2,r<1/\varepsilon})} < C \varepsilon^2.
    \end{equation}
\end{prop}
\begin{proof}
    Let $i_t$ be the injectivity radius of $\tilde{h}_t$. As $\tilde{h}_t$ converges to an ALE gravitational instanton $Y_{\dot{\zeta}}$, we have a uniform bound of $i_t$. 
    In particular, we may cover $Y_{\zeta/\varepsilon^2, r<1/\varepsilon}$ by geodesic balls $\{B(y,r_0)\}$ so that $\{B(y, r_1)\}$ still covers $Y_{\zeta/\varepsilon^2, r<1/\varepsilon}$ for $r_1 <r_0$ (of course, $r_1$ and $r_0$ are independent of $y$ and $t$). 
    Fix $y \in Y_{\zeta/\varepsilon^2,r<1/\varepsilon}$.
    Apply the inequality (\ref{eq:schauder}) to the equation (\ref{eq:linearized rescaled sMA}) restricted to each ball $B(y, r_0)$, we have the following: 
    \begin{equation}
        \left\| \frac{\partial v_{t,s}}{\partial s} \right\|_{C_{\tilde{h}_t}^{2,\alpha}(B(y, r_1 ))} < C \left(\left\|\frac{\partial v_{t,s}}{\partial s}\right\|_{C^0(B(y,r_0))} + \left\|  k_t\right\|_{C_{\tilde{h}_t}^{0,\alpha}(B(y,r_0))} \right) < C\varepsilon^2
        \end{equation}
        for any $y \in Y_{\zeta/\varepsilon^2,r<1/\varepsilon}$ with a constant $C>0$ independent of $y$ and $t$.
\end{proof} 
As an application of the above estimates, we prove the next theorem.

\begin{thm}\label{thm:main2}
    Let $(\mathcal{X}, \mathcal{L}) \to \Delta$ be an analytic family of polarized K3 surfaces over the unit disc $\Delta \subset \mathbb{C}$ such that the fibre $(X_t,L_t)$ on $t \in \Delta$ is smooth for $t \neq 0$ and singular for $t =0$ with ADE singularities. Assume that the family $\mathcal{X} \to \Delta$ admits the minimal simultaneous resolution after taking the base change $\Delta_d \to \Delta$. Then for any smooth function $f$ on $\mathcal{X}$ with its support around a singularity $x_0 \in X_0 \cong 0 \in \mathbb{C}^2/\Gamma$, the function
    \begin{equation*}
        F(t) = \int_{X_t} f_t c_2(\hat{g}_t)
    \end{equation*}
    on $\Delta\backslash \{0\}$ extends to a H\"older continuous function on $\Delta$ with H\"older exponent at least $\frac{1}{d}$.
\end{thm}
\begin{proof}
We take the similar strategy as the proof of theorem \ref{thn:main} to prove the theorem \ref{thm:main2}. Namely, we decompose the function $F$ as follows:
\begin{equation}
    F(t) = \int_{X_t} f_t c_2(\tilde{g}_t) + \int_{X_t} dd^c_t f_t \wedge \tau,
\end{equation}
where $\tau$ is the $(1,1)$ form defined by 
\begin{equation}\label{def:tau2}
    \tau = \frac{1}{4\pi^2 \sqrt{-1}} \int_0^1 (N_s)_{11} (R_s)_{22} - (N_s)_{21} (R_s)_{12} +(N_s)_{22} (R_s)_{11} - (N_s)_{12} (R_s)_{21} ds.
\end{equation}
Here, we use a path $\{g_s\}_{s \in [0,1]}$ from $\tilde{g}_t$ to $\hat{g}_t$ defined by the solution of the equation (\ref{eq:sMA}) (i.e. $g_s = g_{t,s}$) instead of the linear path.
By virtue of the $C^2$ estimates of $\frac{\partial u_{t,s}}{\partial s}$ and $\frac{\partial v_{t,s}}{\partial s}$, we may estimate $\int_{X_t} dd^c_t f_t \wedge \tau$ in a simpler way.

\begin{lem}\label{lem:tau2}
    $\left| \int_{X_t} dd^c_t f_t \wedge \tau \right| \leq C \varepsilon^2$.
\end{lem}
   
\begin{proof}
    By the H\"older inequality and the definition of $\tau$ (\ref{def:tau2}), we have 
    \begin{equation}
    \begin{aligned}
        \left| \int_{X_t} dd^c_t f_t \wedge \tau \right| &\leq C\|dd^c_t f_t \|_{L^2(\tilde{g}_t)} \|\tau \|_{L^2(\tilde{g}_t)}\\
        &\leq C \|\tau \|_{L^2(g_{t,s})}\\
        &\leq C\int_0^1 \| N_s \|_{L^\infty} \|\mathrm{Rm}(g_{t,s}) \|_{L^2(g_{t,s})} ds
     \end{aligned}
    \end{equation}
    for some constant $C >0$ independent of $t$. For the $L^2$-norm of $\mathrm{Rm}(g_{t,s})$, we use the next formula

\begin{lem}[See \cite{BGM}, pp 82]
    Let $(M,g)$ be a $4$-dimensional compact Riemannian manifold. Then the following equality holds:
    \begin{equation}\label{eq:energy}
        \int_{M} \|\mathrm{Rm}(g)\|^2 -4\|\mathrm{Ric}(g)\|^2 + \|\mathrm{Sc}(g)\|^2 \mathrm{Vol}_g = 32\pi^2 e(M),
    \end{equation}
    where $\mathrm{Rm}(g)$ is the total Riemannian curvature tensor, $\mathrm{Ric}(g)$ is the Ricci curvature tensor, $\mathrm{Sc}(g)$ is the scaler curvature and $e(M)$ is the Euler number of $M$.
\end{lem}
In particular, applying the formula (\ref{eq:energy}) to $(X_t, g_{t,s})$, we obtain 
\begin{equation}
    \| \mathrm{Rm}(g_{s,t}) \|_{L^2}^2 = 768\pi^2 + 4\|\mathrm{Ric}(g_{t,s})\|_{L^2}^2 - \|\mathrm{Sc}(g_{t,s})\|^2_{L^2}.
\end{equation}
Then by (\ref{est:Ricci on X-U}), (\ref{est:ricci for 1/2, 3/4}) and (\ref{est:ricci for 2, 1/2}), we have a uniform estimate 
\begin{equation}
\left| 4\|\mathrm{Ric}(g_{t,s})\|_{L^2}^2 - \|\mathrm{Sc}(g_{t,s})\|^2_{L^2} \right| \leq C |t|.
\end{equation}
In particular, 
\begin{equation}
    \|\mathrm{Rm}(g_{t,s}) \|_{L^2} < C
\end{equation}
for a constant $C$ independent of $t$ and $s$.
Recall that 
\begin{equation}
N_s = g^{j\overline{k}}_s \frac{\partial (g_s)_{j\overline{k}}}{\partial s} = g_s^{j\overline{k}}\partial_j\overline{\partial}_k \frac{\partial u_{t,s}}{\partial s},
\end{equation}
where $g_s = \tilde{g}_t + dd^c_t u_{t,s}$.
The $C^2$-estimates for $\frac{\partial u_{t,s}}{\partial s}$ immediately implies the following estimate for $N_s$:
\begin{equation}
    \|N_s \|_{L^\infty(X_t)} \leq C \varepsilon^2.
\end{equation}
Therefore, we obtain the desired estimate for $\int_{X_t} dd^c_t f_t \wedge \tau$.
\end{proof}
    For the integral $\int_{X_t} f_t c_2(\tilde{g}_t)$, we have the following estimate by (\ref{est:c2 on X-U}), (\ref{est:c_2 for 1/2, 3/4}), (\ref{est:c_2 for 2, 1/2}) and (\ref{est: c_2 for 1,2}) (by the parallel calculation as in the section \ref{sec:Main}):
    \begin{equation}\label{eq:F for tilde}
        \int_{X_t} f_t c_2(\tilde{g}_t) = \int_{X_0} f_0 c_2(\hat{g}_0) + f(x_0)e_{orb}(Y_{\dot{\zeta}}) + O(|t|).
    \end{equation}
    Lemma \ref{lem:tau2} and the estimate (\ref{eq:F for tilde}) implies the theorem \ref{thm:main2}.
\end{proof}

\bibliographystyle{plain}
\bibliography{bibliography}

\end{document}